\documentclass{article}
\usepackage{packages}

\title{Quantum $K$-theory of projective spaces and confluence of $q$-difference equations}
\author[1]{Alexis Roquefeuil}
\affil[1]{\small{Kavli IPMU (WPI), UTIAS, The University of Tokyo, Kashiwa, Chiba 277-8583, Japan}}
\date{\today}

\AtEndDocument{%
  \par
  \medskip
  \begin{tabular}{@{}l@{}}%
    \textsc{Alexis Roquefeuil, Kavli IPMU (WPI), UTIAS, The University of Tokyo, Kashiwa, Chiba 277-8583, Japan}\\
    \textit{E-mail address}: \texttt{alexis.roquefeuil@ipmu.jp}
  \end{tabular}}

\begin{document}

\maketitle


\begin{abstract}
    Givental's $K$-theoretical $J$-function can be used to reconstruct genus zero $K$-theoretical Gromov--Witten invariants.
    We view this function as a fundamental solution of a $q$-difference system.
    In the case of projective spaces, we show that we can use the confluence of $q$-difference systems to obtain the cohomological $J$-function from its $K$-theoretic analogue.
    This provides another point of view to one of the statements of Givental--Tonita's quantum Hirzebruch--Riemann--Roch theorem.
    Furthermore, we compute connection numbers in the equivariant setting.
\end{abstract}



\tableofcontents

\section{Introduction}

\subsection{Some context}

\textit{Gromov--Witten invariants} are rational numbers that, in some situations, count the number of curves satisfying some incidence conditions inside a projective algebraic variety.
Let $X$ be a smooth projective variety, and fix $g,n \in \mathbb{Z}_{\geq 0}, d \in H_2(X;\mathbb{Z})$.
Denote by $\overline{\mathcal{M}}_{g,n}(X,d)$ the moduli space of stable maps \cite{Konts_Enumeration}, and let $\left[ \overline{\mathcal{M}}_{g,n}(X,d) \right]^\text{vir}$ be the virtual fundamental class constructed in \cite{BF_intnormalcone}, Definition 5.2.
We recall that this moduli space comes with $n$ evaluation maps $\textnormal{ev}_i : \overline{\mathcal{M}}_{g,n}(X,d) \to X$ and with $n$ (orbifold) vector bundles $\mathcal{L}_i$ called the cotangent line bundles.
We also introduce the cohomological classes $\psi_i := c_1 \left( \mathcal{L}_i \right) \in H^2\left(\overline{\mathcal{M}}_{g,n}(X,d) ; \mathbb{Q} \right)$.

\begin{defin*}[\cite{Konts_Enumeration,BF_intnormalcone}]
    Let $g,n \in \mathbb{Z}_{\geq 0}$, $d \in H_2(X;\mathbb{Z})$. Let $k_1, \dots, k_n \in \mathbb{Z}_{\geq 0}$ be some integers, and let $\alpha_1, \dots \alpha_n \in H^*(X; \mathbb{Q})$. 
    The associated \textit{Gromov--Witten invariant} is defined by the intersection product
    \[
        \langle \psi_1^{k_1} \alpha_1 , \dots , \psi_n^{k_n} \alpha_n \rangle^\textnormal{coh}_{g,n,d}
        = \int_{\left[ \overline{\mathcal{M}}_{g,n}(X,d) \right]^\text{vir}} \bigcup_i \left( \psi_i^{k_i} \cup \text{ev$_i^\star$}(\alpha_i)\right) \in \mathbb{Q}
        ,
    \]
    where $\int_{\left[ \overline{\mathcal{M}}_{g,n}(X,d) \right]^\text{vir}}$ denotes the cap product in cohomology with the virtual fundamental class.
\end{defin*}

More recently, in 2004, Y.-P. Lee defined new invariants by replacing the cohomological constructions in the above definition by their $K$-theoretical analogues.
Denote by $\mathcal{O} ^ {\text{vir}} _ {g,n,d}$ the virtual structure sheaf, constructed in \cite{Lee_qk}, Subsection 2.3.

\begin{defin*}[\cite{Lee_qk}]
    Let $g,n \in \mathbb{Z}_{\geq 0}$, $d \in H_2(X;\mathbb{Z})$. Let $k_1, \dots, k_n \in \mathbb{Z}_{\geq 0}$ be some integers, and let $\phi_1, \dots \phi_n \in K(X)$.
    The associated \textit{$K$-theoretical Gromov--Witten invariant} is given by the Euler characteristic
    \[
        \left\langle
            \mathcal{L}_1^{k_1} \phi_1, \cdots, \mathcal{L}_n^{k_n} \phi_n
        \right\rangle_{g,n,\beta}^{K\textnormal{th}}
        =
        \chi \left(
            \overline{\mathcal{M}}_{g,n}\left(X,d \right);
            \mathcal{O} ^ {\text{vir}} _ {g,n,d}
            \bigotimes_{i=1}^n \mathcal{L}_i^{k_i} \text{ev}_i^*(\phi_i)
        \right)
        \in \mathbb{Z}
        .
    \]
\end{defin*}

A natural question to ask upon reading these two definitions is to understand how these two invariants are related.
An algebraic geometer would rightfully expect them to be related by a Riemann--Roch theorem.
Due to the highly sophisticated geometry of the moduli spaces of stable maps, such formula is not easy to obtain.
In 2014, A. Givental and V. Tonita \cite{Giv_Ton_qk_HRR} found a general result saying that genus zero $K$-theoretical Gromov--Witten invariants can be expressed with genus zero cohomological Gromov--Witten invariants (this result has been extended to all genera in \cite{Giv_PermEquiv_IX}).
However, this formula is very technical and therefore has not seen many applications.
One of its known consequences has been that a key power series expressed with $K$-theoretical Gromov--Witten invariants, called Givental's $K$-theoretical $J$-function, satisfies a system $q$-difference equations (\cite{Giv_Ton_qk_HRR}, Section 9, Theorem, see also \cite{Iri_Mil_Ton_qk}, Proposition 2.12), like it had been verified on some examples, e.g. in \cite{Giv_Lee_qk}, Theorem 2.

\begin{remark*}
    Another approach to obtain a comparison between cohomological and $K$-theoretical Gromov--Witten invariants using derived algebratic geometry has been initiated by A. A. Khan, see \cite{khan_derived_rr}.
\end{remark*}

\subsection{Goal of the article}

The aim of this paper is to propose another point of view to compare $K$-theoretical Gromov--Witten invariants with their cohomological analogues, using the theory of $q$-difference equations.

We will focus on the $q$-difference equations satisfied by Givental's small $K$-theoretical $J$-function of the projective space.
In general, these functional equations satisfy a property called \textit{confluence}, according to which we can take some limit $q \to 1$ of the $q$-difference to obtain a differential equation.
A quick illustration of the confluence of $q$-difference equations is this identity, in which $k \in \mathbb{Z}$,
\[
    \lim_{q \to 1} \frac{\qdeop{Q} - \textnormal{Id}}{q-1} \cdot Q^k
    =
    \lim_{q \to 1} \frac{q^k-1}{q-1} \cdot Q^k
    =
    k Q^k
    =
    Q \partial_Q \cdot Q^k
    .
\]
Therefore, we will say that the $q$-difference operator $\frac{\qdeop{Q} - \textnormal{Id}}{q-1}$ converges formally to the differential operator $Q \partial_Q$.
Our goal is to obtain similar limits for the following data:

\begin{defin*}[\cite{Givental_EquivariantGW},\cite{Giv_PermEquiv_II}]
        Consider $X = \mathbb{P}^N$ with its usual toric action of the torus $T^{N+1} = \left(\mathbb{C}^*\right)^{N+1}$. Let $P=\mathcal{O}_\textnormal{eq}(1) \in K_{T^{N+1}}\left( \mathbb{P}^N \right)$ be the anti-tautological equivariant bundle, and let $\lambda_0, \dots, \lambda_N$ (resp. $\Lambda_0, \dots, \Lambda_N$) be the equivariant parameters in cohomology (resp. $K$-theory).
    \begin{enumerate}[label=(\roman*)]
        \item Let $H = c_1 \left( \mathcal{O}_\textnormal{eq}(1) \right) \in H^2_{T^{N+1}} \left( \mathbb{P}^N; \mathbb{Q} \right)$ be the equivariant hyperplane class.
        Givental's small equivariant cohomological $J$-function of $\mathbb{P}^N$ is given by the expression
        \begin{align*}
            J^{\textnormal{coh},\textnormal{eq}}(z,Q) 
            &=
            Q^{\frac{H}{z}} \sum_{d \geq 0} \frac{Q^d}{\prod_{r=1}^d \left( H - \lambda_0 + rz\right) \cdots \left( H - \lambda_N + rz\right)}
            \\
            &\in
            H_{T^{N+1}}^*\left( \mathbb{P}^N \right) \otimes \mathbb{C}[z,z^{-1}][\![Q]\!],
        \end{align*}
        where
        \[
            Q^{\frac{H}{z}}
            =
            \sum_{k=0}^N \frac{1}{k!} \left( \frac{H}{z} \log(Q) \right)^k
        \]

        \item Givental's small equivariant $K$-theoretical $J$-function is the function
        \begin{align*}
            J^{K\textnormal{th},\textnormal{eq}}(q,Q) 
            &=
            P^{-\ell_q(Q)}
            \sum_{d \geq 0} \frac{Q^d}{\left(q\Lambda_0 P^{-1}, \dots,q\Lambda_N P^{-1} ;q \right)_d}
            \\
            &\in
            K_{T^{N+1}} \left( \mathbb{P}^N \right) \otimes \mathbb{C}[q,q^{-1}][\![Q]\!],
        \end{align*}
        where
        \[
            \left(q\Lambda_0 P^{-1}, \dots,q\Lambda_N P^{-1} ;q \right)_d
            =
            \prod_{i=0}^N \prod_{r=1}^{d}
            (1-q^r \Lambda_i P^{-1}),
        \]
        and $P^{-\ell_q(Q)}$ is some $K$-theoretical function corresponding to the function $Q^{\frac{H}{z}}$, that we will introduce in Definition \ref{qkqde:def_J_fn_eq}.
    \end{enumerate}
\end{defin*}

\begin{prop*}[\cite{Giv_PermEquiv_V,CK_book}]
    \begin{enumerate}[label=(\roman*)]
        \item The cohomological $J$-function $J^{\textnormal{coh},\textnormal{eq}}$ is a solution of the differential equation 
        \begin{equation*}
            (\ref{qkqde:eqn_pde_jh_eq})
            \, : \,
            \left[
                (-\lambda_0 + z Q \partial_Q)
                \cdots
                (- \lambda_N + z Q \partial_Q) - Q
            \right]
            J^{\textnormal{coh},\textnormal{eq}}(z,Q) = 0.
        \end{equation*}

        \item The $K$-theoretical $J$-function $J^{K\textnormal{th},\textnormal{eq}}$ is a solution of the $q$-difference equation
        \begin{equation*}
            (\ref{qkqde:eqn_JK_eq})
            \, : \,
            \left[
                \left(
                    1 - \Lambda_0 \qdeop{Q}
                \right) 
                \cdots
                \left( 
                    1 - \Lambda_N \qdeop{Q} 
                \right) - Q
            \right]
            J^{K\textnormal{th}}(q,Q) 
            = 0.
        \end{equation*}
    \end{enumerate}
\end{prop*}

    Applying the confluence of the $q$-difference equations to this data, we want to first compare the $q$-difference equation satisfied in $K$-theory with the differential equation satisfied in cohomology, then compare the two $J$-functions as solutions of their respective functional equations.
    We would like to expect that the following informal statements hold:
    \begin{enumerate}[label=(\roman*)]
        \item The confluence of the $q$-difference equation $(\ref{qkqde:eqn_JK_eq})$ defines a differential equation $\lim_{q \to 1} (\ref{qkqde:eqn_JK_eq})$ which is the same as the differential equation (\ref{qkqde:eqn_pde_jh_eq}) satisfied by the cohomological $J$-function.

        \item As a solution of the $q$-difference equation $(\ref{qkqde:eqn_JK_eq})$, Givental's $K$-theoretical $J$-function $J^{K\textnormal{th},\textnormal{eq}}$ satisfies
        \[
            \lim_{q \to 1} J^{K\textnormal{th},\textnormal{eq}}
            =
            J^{\textnormal{coh,eq}}.
        \]
    \end{enumerate}

To give a rigorous meaning to these informal identifications, we state the following theorem, which is the first goal of this article:

\begin{thm*}[Theorem \ref{qkqde:thm_confluence_jk_eq}]
    Consider the algebraic torus $T^{N+1} := \left(\mathbb{C}^*\right)^{N+1}$ acting on $X = \mathbb{P}^N$.
    Recall that Equation $(\ref{qkqde:eqn_JK_eq})$ (resp. $(\ref{qkqde:eqn_pde_jh_eq})$) refers to the $q$-difference (resp. differential) equation satisfied by Givental's small equivariant $K$-theoretical (resp. cohomological) $J$-function $J^{K\textnormal{th},\textnormal{eq}}$ (resp. $J^{\textnormal{coh, eq}}$).
    Let $q \in \mathbb{C}, 0<|q|<1$ and $z \in \mathbb{C}^*$.
    Assume that the relation $\Lambda_i = q^\frac{-\lambda_i}{z} \in \mathbb{C}$ holds for all $i \in \{0, \dots, N\}$, and that for $i \neq j$, $\lambda_i-\lambda_j \notin \mathbb{Z}$.
    The following statements hold:
    \begin{enumerate}[label=(\roman*)]
        \item Consider the application $\varphi_{q,z}$ defined by
        \[
            \functiondesc{\varphi_{q,z}}{\mathbb{C}}{\mathbb{C}}{Q}{\left( \frac{z}{1-q} \right)^{N+1} Q}
        \]
        Then, the pullback by $\varphi_{q,z}$ of the $q$-difference equation \eqref{qkqde:eqn_JK_eq} is a confluent $q$-difference equation.
        Moreover, its formal limit when $q \to 1$ is the differential equation \eqref{qkqde:eqn_pde_jh_eq}.
        
        \item Consider the isomorphism of rings
        $
            \gamma_\textnormal{eq} : 
            K_{T^{N+1}} \left( \mathbb{P}^N \right) \otimes \mathbb{C} 
            \to 
            H_{T^{N+1}}^*\left( \mathbb{P}^N, \mathbb{C} \right)
        $
        defined by, for all $i \in \{0, \dots, N\}$
        \[
            \gamma_\textnormal{eq}
            \left(
                \prod_{j \neq i} \frac{1 - \Lambda_i P^{-1}}{1 - \Lambda_i \Lambda_j^{-1}}
            \right)
            =
            \prod_{j \neq i} \frac{H-\lambda_i}{\lambda_j-\lambda_i}
        \]
        Let $
            \mathbb{E}_q
        $
        be the complex torus 
        $
            \mathbb{C}^* / q^\mathbb{Z}
        $
        and let
        $
            \mathcal{M} \left( \mathbb{E}_q \right)
        $
        be the space of meromorphic functions on said complex torus.
        Then, there exists an explicit change of fundamental solution
        $
            P_{q,z} \in \textnormal{GL}_{N+1} \left(
                \mathcal{M} \left( \mathbb{E}_q \right)
            \right)
        $
        such that the fundamental solution $J^{K\textnormal{th},\textnormal{eq}}$ is related to the cohomological $J$-function $J^{\textnormal{coh,eq}}$ by
        \[
            \gamma_\textnormal{eq} \left(
                \lim_{t \to 0}
                P_{q^t,z} \cdot \left(
                    \varphi^*_{q^t,z}
                    J^{K\textnormal{th},\textnormal{eq}}\left(q^t,Q\right)
                \right)
            \right)
            =
            J^{\textnormal{coh,eq}}(z,Q).
        \]
    \end{enumerate}
\end{thm*}

Once this comparison result is established, one could be interested in attempting to compute the (local) monodromy data of this $q$-difference equation and to compare it with the cohomological case (see e.g. \cite{Cotti_Dubrovin_Guzzetti_Helix}).
In the $q$-difference case, the monodromy data is a connection matrix, relating the solution at $Q=0$ given by the $J$-function with a fundamental solution at $Q=\infty$, which we will construct in Proposition \ref{stokes:prop_analytic_fund_sol}.
The second main goal of this article to compute this connection matrix in the equivariant setting, as below.
\begin{thm*}[Theorem \ref{stokes:thm_connection_numbers_for_qk}]
    Let $w=1/Q$ and denote by $g_k$ the fundamental solution at $Q=\infty$ constructed in Proposition \ref{stokes:prop_analytic_fund_sol}.
    Then, the fundamental solutions at $0$ and $\infty$ are related by the identity
    \[
        g_k(w)
        =
        \sum_{j=0
        }^N
        R_{k,j}^{\left[ \lambda; q^{N+1} \right]}(q,w)
        J^{K\textnormal{th, eq}}_{|P = \Lambda_j}\left(q,\frac{1}{w}\right)
    \]
    Where the coefficients $R_{k,j}^{\left[ \lambda; q^{N+1} \right]}$ are some explicit $q^{N+1}$-constant functions.
\end{thm*}
In order to compare with quantum cohomology, one could hope that the limit when $q \to 1$ of this connection matrix would be related to the connection matrix in the cohomology case, c.f. \cite{Cotti_Dubrovin_Guzzetti_Helix}, Theorem 6.7.

\begin{remark*}
    After the appearance of this article as a preprint on the arXiv, similar confluence questions (solutions, connection numbers) were investigated by Y. Wen for quintic threefold in \cite{Wen:qkqde_quinticthreefold}; confluence of the $J$-function has also been proved for any smooth projective variety whose anti-canonical bundle is nef in \cite{Milanov_R:QK_confluence_weak_Fano_and_q_oscillatory}.
\end{remark*}

\subsection{Structure of the article}


The Section \ref{section:QDE} will be a survey on the theory of $q$-difference equations, which the reader might not be familiar with.
The aim of the first two subsections is to introduce the definitions required to understand the statement of the main theorem, as well as the special functions that will be useful to us in quantum $K$-theory.
Then, in a last subsection, we will explain the confluence properties of $q$-difference equations in the regular singular case.

In the Section \ref{section:QK_QDE}, the reader should now have the necessary background to understand the statement of the Main Theorem.
In the first subsection, we will recall the definitions of Givental's equivariant $J$-functions, whose expressions are obtained by using virtual localisation theorems. Then, we will give their functional equations.
In the second subsection, we state the Main Theorem and give its proof.
Our proof is split in two parts: first we check the confluence of the $q$-difference equation, then we check the confluence of the $K$-theoretical $J$-function as a solution of the confluent $q$-difference equation.
In the third subsection, we will explain what happens when one tries to adapt the main theorem for non equivariant $J$-functions.

In the Section \ref{section:qmonodromy}, we will compute the $q$-monodromy of our $q$-difference equation in the equivariant case.
In the first subsection, we construct another fundamental solution, this time at $Q=\infty$, at which the $q$-difference equation is irregular singular.
In the second subsection, we prove a base change formula from the $J$-function to this new fundamental solution, obtaining connection numbers in the equivariant case.
Unfortunately, we are not able to prove a non equivalent analogue of these connection numbers, but we will be able to conjecture some formula.

\subsection{Acknowledgements}

Part of the results presented in this article were obtained while the author was a Ph.D student. They would like to acknowledge his advisor, \'Etienne Mann for his advice, guidance, suggestions and friendliness.
The remaining results were obtained while the author was a JSPS International Research Fellow (Standard program at Kavli IPMU, University of Tokyo).
 
This work is supported by Universit\'{e} d'Angers (imputation budg\'{e}taire A900210), France; and the World Premier International Research Center Initiative (WPI Initiative), Ministry of Education, Culture, Sports, Science and Technology, Japan.
The author was partially supported by Agence Nationale de la Recherche's projects ANR-13-IS01-0001 (SISYPH) and ANR-17-CE40-0014 (CatAG); and by Japanese Society for the Promotion of Science's Kakenhi JP19F19802.
\section{$q$-difference equations}\label{section:QDE}

This section is structured in three subsections.
In the first subsection, we give some introductory definitions regarding $q$-difference equations.
The second section is dedicated to the resolution of regular singular $q$-difference equations.
The last section deals with confluence of regular singular $q$-difference equations.

\subsection{General definitions}

In this subsection we recall general notions of the theory of $q$-difference equations from the analytical point of view.

\begin{defin}\label{qde:def_q_difference_linear_sys}
Let $\mathcal{M}(\mathbb{C})$ be the field of meromorphic functions on $\mathbb{C}$.
Fix $q\in \mathbb{C}, |q|<1$ and $n \in \mathbb{Z}_{>0}$. Let $q^{Q \partial_Q}$ be the $q$-difference operator acting on functions $f: \mathbb{C} \to \mathbb{C}$ by $\left(q^{Q \partial_Q} f\right)(Q) = f(q Q)$. A \textit{linear $q$-difference system} is a functional equation
\begin{equation*}
    q^{Q \partial_Q} X_q(Q) = A_q(Q) X_q(Q),
\end{equation*}
where $X_q$ is a column vector of $n$ functions of input $Q$, and $A \in \text{M}_n (\mathcal{M}(\mathbb{C}))$.
The \textit{rank} of this $q$-difference system is defined to be the rank of the matrix $A_q$.
\end{defin}

From now on we will work locally at $Q=0$.
More precisely, we will look for solutions in the space $\mathbb{C}\left\{Q,Q^{-1}\right\}$ of Laurent series that are convergent on a punctured disk centered at $Q=0$.
The definitions and the results below would also hold for $Q=\infty$ after replacing $Q$ with $Q^{-1}$.

\begin{defin}
    Let $(\ddagger_q) : q^{Q \partial_Q} X_q(Q)=A_q(Q) X_q(Q)$ be a $q$-difference system, with $A_q \in \textnormal{M}_n \left( \mathcal{M}(\mathbb{C}) \right)$.
    We define the \textit{solution space} of this $q$-difference equation by
    \[
        \textnormal{Sol}\left( \ddagger_q \right)
        =
        \left\{
            X_q \in \left( \mathbb{C}\left\{Q,Q^{-1}\right\} \right)^n
            \, \middle| \,
            q^{Q \partial_Q} X_q(Q)=A_q(Q) X_q(Q)
        \right\}.
    \]
\end{defin}

\begin{example}[$q$-constants]
    Consider the $q$-difference equation 
    \[
    \qdeop{Q} f_q(Q) = f_q(Q).
    \]
    Constant functions are obvious solutions to this $q$-difference equation.
    Denote by $q^\mathbb{Z}$ the multiplicative group $q^\mathbb{Z} := \{ q^k \, | \, k \in \mathbb{Z} \}$ and choose $\tau \in \mathbb{H} \subset \mathbb{C}$ such that $q = e^{2i \pi \tau}$.
    The meromorphic solutions of this $q$-difference equation can be identified with meromorphic functions on the torus $\mathbb{C}^* / q^\mathbb{Z}$, where
    the action is given by the multiplication $q^k \cdot z = q^k z$ and the complex (torus) structure comes from the exponential $\left( z \mapsto \exp(2i \pi z) \right)$, as in the diagram below.
    \begin{center}
        \begin{tikzcd}
            \mathbb{C}
                \arrow[r,"\exp"]
                \arrow[d]
            &
            \mathbb{C}^*
                \arrow[d]
            \\
            \frac{\mathbb{C}}{\mathbb{Z} + \tau \mathbb{Z}}
                \arrow[r,"\widetilde{\exp}","\sim"']
            &
            \frac{\mathbb{C}^*}{q^\mathbb{Z}}
        \end{tikzcd}
    \end{center}
    Solutions to this $q$-difference equation will be called \textit{$q$-constants}.
\end{example}

\begin{notation}
    We denote by $\mathbb{E}_q$ the complex torus $\mathbb{E}_q := \mathbb{C}^* / q^\mathbb{Z}$.
\end{notation}

The space $\mathcal{M} \left( \mathbb{E}_q \right)$ of meromorphic functions on the complex torus $\mathbb{C}^* / q^\mathbb{Z}$ plays a role for $q$-difference equations similar to the space of constant functions $\mathbb{C}$ for differential equations.
Because the $q$-difference operator $\qdeop{Q}$ is an automorphism of complex functions, we have the following property:

\begin{prop}[\cite{Har_Sau_Sin_book}, p.116]
    The set $\textnormal{Sol}\left( \ddagger_q \right)$ has a structure of $\mathcal{M} \left( \mathbb{E}_q \right)$-vector space.
\end{prop}

\begin{defin}
    Let $q^{Q \partial_Q} X_q(Q)=A_q(Q) X_q(Q)$ be a $q$-difference system of rank $n \in \mathbb{Z}_{>0}$.
    A \textit{fundamental solution} of this system is an invertible matrix
    $\mathcal{X}_q \in \text{GL}_n\left(\mathbb{C}\left\{Q,Q^{-1}\right\}\right)$
    such that $q^{Q \partial_Q} \mathcal{X}_q(Q) = A_q(Q) \mathcal{X}_q(Q)$.
\end{defin}

\begin{defin}
    Let $q^{Q \partial_Q} X_q(Q) = A_q(Q) X_q(Q)$ be a $q$-difference system.
    Consider a matrix $F_q \in \text{GL}_n\left( \mathbb{C}\left\{Q,Q^{-1}\right\} \right)$.
    The \textit{gauge transform} of the matrix $A_q$ by the \textit{gauge transformation} $F_q$ is defined to be the matrix
    \[
        F_q \cdot [A_q]
        :=
        \left( q^{Q \partial_Q} F_q \right) A_q F_q^{-1}.
    \]
    A second $q$-difference system $q^{Q \partial_Q} X_q(Q) = B_q(Q) X_q(Q)$ is said to be \textit{equivalent by gauge transform} to the first one if there exists a matrix $F_q \in \text{GL}_n\left( \mathbb{C}\left\{Q,Q^{-1}\right\} \right)$ such that
    \[
        B_q = F_q \cdot [A_q].
    \]
\end{defin}

\begin{defin}
    Let $(\ddagger_q) : q^{Q \partial_Q} X_q(Q)=A_q(Q) X_q(Q)$ be a $q$-difference system and let $\varphi_q : \mathbb{C} \to \mathbb{C}$ be an isomorphism.
    The \textit{$q$-pullback} $(\varphi_q^* \ddagger_q)$ of $(\ddagger_q)$ by $\varphi_q$ is the $q$-difference system given by
    \[
        \varphi_q^*(\ddagger_q) \, : \, \qdeop{Q} X_q(Q) = A_q(\varphi_q^{-1}(Q)) X_q(Q).
    \]
\end{defin}

\begin{defin}
    A system $q^{Q \partial_Q} X_q(Q) = A_q(Q) X_q(Q)$ is \textit{regular} if $A_q(0)$ is diagonal and if its eigenvalues are of the form $q^k$ for $k \in \mathbb{Z}_{\geq 0}$.
\end{defin}

\begin{defin}\label{qde:def_regular_singular_qde}
    A system $q^{Q \partial_Q} X_q(Q) = A_q(Q) X_q(Q)$ is said to be \textit{regular singular at $Q=0$} if there exists a $q$-gauge transform $P_q \in \text{GL}_n\left( \mathbb{C}\left\{Q,Q^{-1}\right\} \right)$ after which the matrix $A_q$ evaluated at $Q=0$ is well-defined and invertible, i.e. $\left( P_q \cdot[A_q] \right) (0) \in \text{GL}_n(\mathbb{C})$.
\end{defin}

\subsection{Fundamental solution for regular singular $q$-difference systems}

We will now mention the results regarding the fundamental solution of regular singular $q$-difference equations.
Like in the differential case, if a $q$-difference equation is regular singular at $0$, then its solutions exhibit polynomial growth at $0$.
This statement is made explicit in Proposition \ref{qde:prop_sol_have_polynomial_growth}.
The practical use of this part is to introduce various special functions related to the theory of $q$-difference equations, which will also appear in the next section dealing with quantum $K$-theory: Definitions \ref{qde:def_qpoch} and \ref{qde:def_q_logarithm}.

\begin{defin}\label{qde:def_qpoch}
    The \textit{$q$-Pochhammer symbol} is the complex function defined for $d \in \mathbb{Z}_{\geq 0}$ by
    \begin{align*}
        (Q;q)_0 &:= 1, \\
        (Q;q)_d &:= \prod_{r=0}^{d-1} (1-q^r Q), \\
        (Q;q)_\infty &:= \prod_{r \geq 0} (1-q^r Q).
    \end{align*}
\end{defin}

\begin{defin}[\cite{Mum_tata_1}]\label{qde:def_theta}
    \textit{Jacobi's theta function} $\theta_q$ is the complex function defined by the convergent Laurent series
    \[
        \theta_q(Q) := \sum_{d \in \mathbb{Z}} q^\frac{d(d-1)}{2} Q^d
        \in
        \mathbb{C}\{\!\{Q,Q^{-1}\}\!\}.
    \]
\end{defin}

\begin{prop}\label{qde:prop_theta_qde}
    Jacobi's theta function $\theta_q$ is a solution of the $q$-difference equation
    \[
        \qdeop{Q} \theta_q(Q) = \frac{1}{Q} \theta_q(Q).
    \]
\end{prop}

\begin{defin}[\cite{Sauloy_qde_regsing}]\label{qde:def_q_logarithm}
    The \textit{$q$-logarithm} is the function $\ell_q \in \mathcal{M}(\mathbb{C}^*)$ defined by
    \[
        \ell_q(Q) := \frac{-Q \theta_q'(Q)}{\theta_q(Q)}.
    \]
\end{defin}

\begin{lemma}[\cite{Sauloy_qde_regsing}]\label{qde:prop_ell_q_is_a_q_log}
    The function $\ell_q$ is a solution of the $q$-difference equation
    \[
        \qdeop{Q} \ell_q(Q) = \ell_q(Q) + 1.
    \]
\end{lemma}

\begin{defin}[\cite{Sauloy_qde_regsing}]
    A regular singular $q$-difference system $q^{Q \partial_Q} X = A_q(Q) X$ is said to be \textit{non resonant} if any couple of two different eigenvalues $\lambda_i \neq \lambda_j$ of the matrix $A_q(0)$ satisfies the condition $\lambda_i \lambda_j^{-1} \notin q^\mathbb{Z}$.
\end{defin}

\begin{thm}[\cite{Sauloy_qde_regsing}, Subsection 1.1.4]\label{qde:thm_qde_frob}
    Let $q^{Q \partial_Q} X_q(Q) = A_q(Q) X_q(Q)$ be a regular singular $q$-difference system which is non resonant.
    There exists a fundamental solution of $\mathcal{X}_q \in \text{GL}_n\left(\mathbb{C}\left\{Q,Q^{-1}\right\}\right)$ of this $q$-difference equation expressed with the function $\ell_q$ of Definition \ref{qde:def_q_logarithm}.
\end{thm}

\begin{prop}[\cite{Har_Sau_Sin_book}, Theorem 3.1.7, p.127]\label{qde:prop_sol_have_polynomial_growth}
    Let $\qdeop{Q} X_q(Q) = A_q(Q) X_q(Q)$ be a non resonant regular singular $q$-difference system of rank $n$.
    For $i \in \{ 1, \dots, n\}$, denote by $X_{(i)}$ the $i^\textnormal{th}$ column of the fundamental solution given by Theorem \ref{qde:thm_qde_frob}.
    Choose any $\nu \in \mathbb{C}, |\nu|=1$ such that the function below is locally well defined:
    \[
        \functiondesc
        {f_{i,\nu}}
        {(\mathbb{R},+\infty)}
        {(\mathbb{C}^*,0)}
        {t}
        {X_{(i)}(\nu q^t)}.
    \]
    Then, the function $f_{i,\nu}$ has polynomial growth at $t=+\infty$.
\end{prop}

\subsection{Confluence of regular singular $q$-difference equation}

In this subsection we introduce Sauloy's confluence phenomenon. One of the main ingredient is the following asymptotic for the $q$-logarithm $\ell_q$.

\begin{notation}
    Let $\lambda_q \in \mathbb{C}^*$ be some non zero complex number.
    We will call $q$-spiral the set $\lambda_q q^\mathbb{R}:= \left\{ \lambda_q q^t \, \middle| \, t \in \mathbb{R} \right\} \subset \mathbb{C}^*$.
    Note that its complementary in $\mathbb{C}^*$ is simply connected.
\end{notation}

\begin{prop}[\cite{Sauloy_qde_regsing}, Subsections 3.1.3 and 3.1.4]\label{qde:prop_specfns_limits}
    Fix $q_0 \in \mathbb{C}^*, |q_0|<1$, let $q(t)=q_0^t, t \in (0,1]$. 
        Denote by $\log$ the determination of the logarithm on $\mathbb{C}^*- (-1)q_0^\mathbb{R}$ such that $\log(1) = 0$.
        We have the uniform convergence, on any compact of $\mathbb{C}^*- (-1)q_0^\mathbb{R}$,
        \[
            \lim_{t \to 0} (q(t)-1) \ell_{q(t)}(Q) = \log(Q).
        \]
        
\end{prop}

\begin{remark}
    Let $\delta_q = \frac{\qdeop{Q}-\textnormal{Id}}{q-1}$. We recall that the formal limit of this $q$-difference operator is the differential operator $Q \partial_Q$.
    A motivation to consider the function $(q-1)\ell_q$ instead of the usual $q$-logarithm $\ell_q$ is that we have
    \[
        \delta_q
        \begin{pmatrix}
            1               \\
            (q-1)\ell_q(Q)
        \end{pmatrix}
        =
        \begin{pmatrix}
            0   &   0   \\
            1   &   0   \\
        \end{pmatrix}
        \begin{pmatrix}
            1               \\
            (q-1)\ell_q(Q)
        \end{pmatrix}
    \]
    Notice that the formal limit of this $q$-difference system is the differential system satisfied by the logarithm, while the matrix associated to the $q$-difference equation of $\ell_q$ has no limit when $q \to 1$.
\end{remark}

\subsubsection*{Confluence of the $q$-difference equation}

The first step is to define which regular singular $q$-difference equations have a well behaved formal limit.

\begin{defin}[\cite{Sauloy_qde_regsing} Section 3.2]\label{qde:def_confluentsys}
    Let $q_0 \in \mathbb{C}, |q_0|<1$, and set $q(t)=q_0^t$, for $t \in (0,1]$.
    A regular singular non resonant $q$-difference system $q^{Q \partial_Q} X = A_q(Q) X$ is said to be \textit{confluent} if it satisfies the four conditions below.
    Set $B_q(Q) = \frac{A_q(Q) - \text{Id}}{q-1}$, whose coefficients have poles in the input $Q$ that we will denote by $Q_1(q), \dots, Q_k(q)$. We require that
    \begin{enumerate}[label=(\roman*)]
        \item The $q$-spirals satisfy $\bigcap_{i=1}^k Q_i(q_0) q_0^\mathbb{R} = \varnothing$.

        \item There exists a matrix $\Tilde{B} \in \text{GL}_n \left(\mathbb{C}(Q)\right)$ such that
        \[
            \lim_{t \to 0} B_{q(t)} = \Tilde{B},
        \]
        uniformly in $Q$ on any compact of $\mathbb{C}^* - \bigcup_{i=0}^k Q_i q_0^\mathbb{R}$, where $Q_0 := 1$.
        
        \item This limit defines a regular singular, non resonant differential system
        \begin{equation*}
            Q \partial_Q \widetilde{X} = \widetilde{B} \widetilde{X} ,
        \end{equation*}
        with distinct singularities $\widetilde{Q}_i = \lim_{t \to 0} Q_i(q)$.
        
        \item There exists Jordan decompositions $B_{q(t)}(0)={P_{q(t)}}^{-1} J_{q(t)} P_{q(t)} \, , \, t \in (0,1]$ and $\widetilde{B}(0)=\widetilde{P}^{-1} \widetilde{J} \widetilde{P}$;
        such that \[ \lim_{t \to 0} P_{q(t)} = \widetilde{P}. \]
    \end{enumerate}
\end{defin}

\subsubsection*{Confluence of the solutions}

\begin{thm}[\cite{Sauloy_qde_regsing}, Section 3.3]\label{qde:thm_confluence_of_sol_with_initial_cond}
    Let $q_0 \in \mathbb{C}, |q_0|<1$, and set $q(t)=q_0^t$, for $t\in(0,1]$.
    Consider a regular singular confluent $q$-difference system \[q^{Q \partial_Q} X_q(Q) = A_q(Q) X_q(Q),\] whose formal limit when $q \to 1$ is the differential system $Q \partial_Q X(Q) = \widetilde{B}(Q) X(Q)$.
    
    We assume that there exists a vector $X_0 \in \mathbb{C}^n - \{0\}$, independent of $q$, such that $A_{q(t)} X_0 = X_0$ for all $t \in (0,1]$.
    We also assume that we have a solution $X_q$ of the $q$-difference system satisfying the initial condition $X_q(0)=X_0$.
    
    Let $\widetilde{X}$ be the unique solution of $Q \partial_Q X(Q) = \widetilde{B}(Q) X(Q)$ satisfying the initial condition $\widetilde{X}(0)=X_0$.
    Then,
    \[
        \lim_{t \to 0} X_{q(t)}(Q) = \widetilde{X}(Q),
    \]
    uniformly in $Q$ on any compact of $\mathbb{C}^* - \bigcup_{i=0}^k Q_i q_0^\mathbb{R}$.
\end{thm}

Applying this theorem to the fundamental solution given by Theorem \ref{qde:thm_qde_frob} gives the corollary below.

\begin{coro}[\cite{Sauloy_qde_regsing}, Subsections 3.2.4 and 3.4]\label{qde:coro_confluence_for_fund_sols}
    Let $q_0 \in \mathbb{C}, |q_0|<1$, and set $q(t)=q_0^t$, for $t\in(0,1]$.
    Let $q^{Q \partial_Q} X_q(Q) = A_q(Q) X_q(Q)$ be a confluent regular singular $q$-difference system, whose formal limit when $q \to 1$ is the differential system $Q \partial_Q X(Q) = \widetilde{B}(Q) X(Q)$.
    Denote by $\mathcal{X}_{q_0}$ the fundamental solution of the $q$-difference given by Theorem \ref{qde:thm_qde_frob}.
    Then, the limit matrix $\lim_{t \to 0} \mathcal{X}_{q(t)}$ is a fundamental solution of the differential equation $Q \partial_Q X(Q) = \widetilde{B}(Q) X(Q)$.
\end{coro}

However, not any fundamental solution of a confluent $q$-difference system has immediately a well defined limit when $t \to 0$, e.g. the constant function $(1-q)^{-1}$ is a $q$-constant.
We therefore introduce the following definition:

\begin{defin}\label{qde:def_confluent_solution}
    Let $\qdeop{Q} X = A_q X$ be a confluent $q$-difference system. A fundamental solution $\mathcal{X}_q$ is \textit{confluent} if $\lim_{t \to 0} \mathcal{X}_{q^t}$ exists and is a fundamental solution of the formal limit of the confluent $q$-difference system $\qdeop{Q} X = A_q X$.
\end{defin}
\section{Confluence for quantum $K$-theory of projective spaces}\label{section:QK_QDE}

\subsection{Equivariant $J$-functions}

\subsubsection*{Definitions.}

Let $N \in \mathbb{Z}_{>0}$ be some positive integer and consider the projective space $X = \mathbb{P}^N$ with the action of the torus $T^{N+1} := (\mathbb{C}^*)^{N+1}$ given by
\[
    (\lambda_0, \dots, \lambda_N) \cdot [z_0 : \cdots : z_N] = [\lambda_0 z_0 : \cdots : \lambda_N z_N].
\]
The elementary representations, indexed by $i \in \{0, \dots, N\}$,
\[
    \functiondesc
    {\rho_i}
    {(\mathbb{C}^*)^{N+1}}
    {\mathbb{C}^*}
    {(t_0,\dots,t_N)}
    {t_i},
\]
define $N+1$ classes in equivariant $K$-theory $\Lambda_0,\dots,\Lambda_N \in K_{T^{N+1}}(\textnormal{pt})$, where $- \Lambda_i$ is the line bundle on the point with the action of the group $T^{N+1}$ given by $\rho_i$.
Denote by $P=\mathcal{O}_\textnormal{eq}(1) \in K_{T^{N+1}}\left( \mathbb{P}^N \right)$ the anti-tautological equivariant bundle, $H=c_1( \mathcal{O}_\textnormal{eq}(1))$ the equivariant hyperplane class and $\lambda_i = c_1 (\Lambda_i) \in H^2_{T^{N+1}}(\textnormal{pt})$.
We recall that we have
\begin{align*}
    K_{T^{N+1}} \left( \mathbb{P}^N \right)
    &\simeq
    \left.
        \mathbb{Z}[\Lambda_0^{\pm 1},\dots,\Lambda_N^{\pm 1}][P^{\pm 1}]
    \middle/
        \left(
            (1 - \Lambda_0 P^{-1}) \cdots (1 - \Lambda_N P^{-1})
        \right)
    \right. ,
    \\
    H_{T^{N+1}}^*\left(\mathbb{P}^N ; \mathbb{Q} \right)
    &\simeq
    \left.
        \mathbb{Q}[\lambda_0,\dots,\lambda_N][H]
    \middle/
        \left( (H-\lambda_0) \cdots (H - \lambda_N) \right)
    \right. .
\end{align*}
A basis of the equivariant $K$-theory $K_{T^{N+1}} \left( \mathbb{P}^N \right)$ is given by the classes indexed by $i \in \{0, \dots, N\}$
\[
    \eta_i = \prod_{j \neq i} \frac{1 - \Lambda_j P^{-1}}{1 - \Lambda_j \Lambda_i^{-1}} \in K_{T^{N+1}} \left( \mathbb{P}^N \right).
\]

\begin{defin}[\cite{CK_book}, Subsection 11.2.3]
    \textit{Givental's small equivariant cohomological $J$-function} of the projective space $\mathbb{P}^N$ is the function defined by
    \begin{align*}
        J^{\textnormal{coh},\textnormal{eq}}(z,Q) 
        &=
        Q^{\frac{H}{z}} \sum_{d \geq 0} \frac{Q^d}{\prod_{r=1}^d \left( H - \lambda_0 + rz\right) \cdots \left( H - \lambda_N + rz\right)}
        \\
        &\in
        H_{T^{N+1}}^*\left( \mathbb{P}^N \right) \otimes \mathbb{C}[\![z,z^{-1}]\!].
    \end{align*}
\end{defin}

\begin{remark}
    The reader familiar with Gromov--Witten theory may notice several abuses in this definition of the $J$-function.
    The proper way to define them would be from the fundamental of the quantum $\mathcal{D}$-module (see e.g. \cite{CK_book}, Equation 10.28 and \cite{Iri_Mil_Ton_qk}, Definition 2.4).
    We also confuse the $I$-function and the $J$-function for projective spaces due to the triviality of the mirror map for complex projective spaces.
\end{remark}

\begin{prop}
    The cohomological $J$-function $J^{\textnormal{coh},\textnormal{eq}}$ is a solution of the differential equation
    \begin{equation}\label{qkqde:eqn_pde_jh_eq}
        \left[(-\lambda_0 + z Q \partial_Q) \cdots (- \lambda_N + z Q \partial_Q) - Q\right]
        J^{\textnormal{coh},\textnormal{eq}}(z,Q) = 0.
    \end{equation}
\end{prop}

\begin{defin}[\cite{Giv_PermEquiv_II}]\label{qkqde:def_J_fn_eq}
    \textit{Givental's equivariant small $K$-theoretical $J$-function} of the projective space $\mathbb{P}^N$ is the function defined by
    \begin{equation}\label{qkqde:eqn_J_def_eq}
        J^{K\textnormal{th},\textnormal{eq}}(q,Q) 
        =
        P^{-\ell_q(Q)}
        \sum_{d \geq 0} \frac{Q^d}{\left(q\Lambda_0 P^{-1}, \dots,q\Lambda_N P^{-1} ;q \right)_d},
    \end{equation}
    where
    \begin{align*}
        \left(q\Lambda_0 P^{-1}, \dots,q\Lambda_N P^{-1} ;q \right)_d
        &=
        \prod_{i=0}^N
        \left(q\Lambda_i P^{-1} ;q \right)_d,
        \\
        P^{-\ell_q(Q)}
        &=
        \sum_{i=0}^N
        \Lambda_i^{-\ell_q(Q)}
        \eta_i.
    \end{align*}
\end{defin}

\begin{prop}[\cite{Giv_PermEquiv_V}]
    The $K$-theoretical $J$-function $J^{K\textnormal{th},\textnormal{eq}}$ is a solution of the $q$-difference equation, which is regular singular at $Q=0$ :
    \begin{equation}\label{qkqde:eqn_JK_eq}
        \left[ \left( 1 - \Lambda_0 \qdeop{Q} \right) \cdots \left( 1 - \Lambda_N \qdeop{Q} \right) - Q \right]  J^{K\textnormal{th,eq}}(q,Q) = 0.
    \end{equation}
\end{prop}

\begin{remark}[On the inputs $z$ and $q$]
    Geometrically, the input $q$ (resp. $z$) can be understood as a generator of the $\mathbb{C}^*$-equivariant $K$-theory (resp. cohomology) of the point, see Section 2.6 of \cite{Iri_Mil_Ton_qk} for details.
    Then, these generators are related by the identity $z = - c_1(q) \in H^*_{\mathbb{C}^*}(\textnormal{pt})$.
\end{remark}

\subsubsection*{A remark on the choice of the function $P^{-\ell_q(Q)}$.}

This part will be a comparison between the $K$-theoretical function $P^{-\ell_q(Q)}$ we have introduced in the Definition \ref{qkqde:def_J_fn_eq} and the usual $q$-characters $e_{q,\lambda_q}$ that appear in the analytic theory of regular singular $q$-difference equations.
This optional part is independent of the main theorem. The reader may want to skip to Subsection \ref{qkqde:subsection_main_theorem_equivariant}.

\begin{prop}
    The $K$-theoretical function defined by
    \[
    P^{-\ell_q(Q)}
        :=
        \sum_{i=0}^N
        \Lambda_i^{-\ell_q(Q)}
        \eta_i
    \]
    is a solution of the $K$-theoretically valued $q$-difference equation
    \[
        \qdeop{Q} f(Q) = P^{-1} f(Q).
    \]
\end{prop}

Functions that satisfy such $q$-difference equations are called \textit{$q$-characters}.
Recall that Jacobi's theta function $\theta_q$, by Proposition \ref{qde:prop_theta_qde}, is a solution of the $q$-difference equation $\qdeop{Q} \theta_q(Q) = Q^{-1} \theta_q(Q)$.
A common example of a $q$-character is the following function:

\begin{defin}[\cite{Sauloy_qde_regsing}, Subsection 1.1.2]
    Let $\lambda_q \in \mathbb{C}^*$. The corresponding \textit{$q$-character} is the function $e_{q, \lambda_q}$ defined by
    \[
        e_{q,\lambda_q}(Q) = \frac{\theta_q(Q)}{\theta_q(\lambda_q Q)} \in \mathcal{M} \left( \mathbb{C}^* \right).
    \]
\end{defin}

\begin{prop}[\cite{Sauloy_qde_regsing}, Subsection 1.1.2]
    Let $\lambda_q \in \mathbb{C}^*$.
    The function $e_{q,\lambda_q}$ is a solution of the $q$-difference equation
    \[
        \qdeop{Q} e_{q,\lambda_q}(Q) = \lambda_q e_{q,\lambda_q}(Q)
    \]
\end{prop}

\begin{remark}
    For the equivariant $K$-theoretical $J$-function, instead of using the function $P^{-\ell_q(Q)}$, it would have been possible to introduce the function $e_{q,P^{-1}}$ defined by
    \[
        e_{q,P^{-1}}(Q)
        =
        \sum_{i=0}^N
        e_{q,\Lambda_i^{-1}}(Q)
        \eta_i
    \]
    We chose the former to have a better basis decomposition when considering the non equivariant limit $\Lambda_i \to 1$.
    Indeed, the non equivariant $K$-theory of $\mathbb{P}^N$ is given by
    $K\left( \mathbb{P}^N \right) \simeq \mathbb{Z}[P,P^{-1}] \left/\left( \left(1-P^{-1} \right)^{N+1} \right) \right.$.
    Let us write
    \[
        \binom{\ell_q(Q)}{k} = \frac{1}{k!} \prod_{r=0}^{k-1} (\ell_q(Q)-r)
    \]
    The function $P^{-\ell_q(Q)}$ has the decomposition in the previous basis of the non equivariant $K$-theory
    \[
        P^{-\ell_q(Q)}
        =
        \left( 1 - \left(1-P^{-1}\right) \right)^{\ell_q(Q)}
        =
        \sum_{k \geq 0} (-1)^k \binom{\ell_q(Q)}{k} \left( 1 - P^{-1} \right)^k.
    \]
    Let us point out that the family $\left(1,\ell_q(Q),\dots,\ell_q(Q)^N\right)$ is linearly independent over the field of $q$-constants $\mathcal{M}\left(\mathbb{E}_q\right)$, see \cite{roquefeuil_thesis}, Lemma VI.1.1.10.
    This function has to be compared with the infinite product below, whose decomposition in our basis of the non equivariant $K$-theory is much more technical,
    \[
        e_{q, P^{-1}}(Q) := \theta_q(Q) \theta_q(P^{-1} Q)^{-1}.
    \]
    Therefore, when defining the $J$-function, we decided to use the function $P^{-\ell_q(Q)}$ instead of the usual $q$-character $e_{q, P^{-1}}(Q)$.
\end{remark}

\subsection{Confluence of the $J$-function}\label{qkqde:subsection_main_theorem_equivariant}

We begin by making a remark on the equivariant parameters to justify a relation that will appear in our statement of the confluence of the $K$-theoretical $J$-function.

\begin{remark}
    Recall that we have $z = -c_1(q) \in H_{\mathbb{C}^*}^*(\textnormal{pt})$ and $\lambda_i = c_1(\Lambda_i) \in H_{\mathbb{T}^{N+1}}^*(\textnormal{pt})$.
    The morphism $f : T^{N+1} \to \mathbb{C}^*$ given by $f(w_0, \dots, w_N) = w_0 \cdots w_N$ induces morphisms $f_{K\textnormal{th}} : K_{\mathbb{C}^*}(\textnormal{pt}) \to K_{T^{N+1}}(\textnormal{pt})$ and $f_\textnormal{coh} : H^*_{\mathbb{C}^*}(\textnormal{pt}) \to H^*_{T^{N+1}}(\textnormal{pt})$.
    We have the relation in the equivariant cohomology $H_{T^{N+1}}^* \left( \mathbb{P}^N \right)$, up to degree 2 terms
    \[
        \textnormal{ch}(\Lambda_i) = \textnormal{ch}( f_{K\textnormal{th}}(q) )^{-\frac{\lambda_i}{f_\textnormal{coh}(z)}}.
    \]
\end{remark}

\subsubsection*{Statement.}

\begin{thm}\label{qkqde:thm_confluence_jk_eq}
    Consider the algebraic torus $T^{N+1} = \left(\mathbb{C}^*\right)^{N+1}$ acting on $X = \mathbb{P}^N$.
    Recall that $(\ref{qkqde:eqn_JK_eq})$ (resp. $(\ref{qkqde:eqn_pde_jh_eq})$) denotes the $q$-difference (resp. differential) equation satisfied by Givental's small equivariant $K$-theoretical (resp. cohomological) $J$-function $J^{K\textnormal{th},\textnormal{eq}}$ (resp. $J^{\textnormal{coh,eq}}$).
    Assume that the relation $\Lambda_i = q^\frac{-\lambda_i}{z} \in \mathbb{C}$ holds for all $i \in \{0, \dots, N\}$, and that for $i \neq j$, $\lambda_i-\lambda_j \notin \mathbb{Z}$.
    Let $q \in \mathbb{C}, 0<|q|<1$ and $z \in \mathbb{C}^*$.
    The following statements hold:
    \begin{enumerate}[label=(\roman*)]
        \item Consider the map $\varphi_{q,z}$ defined by
        \[
            \functiondesc{\varphi_{q,z}}{\mathbb{C}}{\mathbb{C}}{Q}{\left( \frac{z}{1-q} \right)^{N+1} Q.}
        \]
        Then, the pullback by $\varphi_{q,z}$ of the $q$-difference equation \eqref{qkqde:eqn_JK_eq} is a confluent $q$-difference equation.
        Moreover, its formal limit when $q \to 1$ is the differential equation \eqref{qkqde:eqn_pde_jh_eq} satisfied by the cohomological $J$-function.
        
        \item Let $
            \mathbb{E}_q
        $
        be the complex torus 
        $
            \mathbb{C}^* / q^\mathbb{Z}
        $
        and
        $\mathcal{M} \left( \mathbb{E}_q \right)$
        be the space of meromorphic functions on said complex torus. Consider the isomorphism of rings
        $
            \gamma_\textnormal{eq} : 
            K_{T^{N+1}} \left( \mathbb{P}^N \right) \otimes \mathbb{C} 
            \to 
            H_{T^{N+1}}^*\left( \mathbb{P}^N, \mathbb{C} \right)
        $
        defined by, for all $i \in \{0, \dots, N\}$
        \[
            \gamma_\textnormal{eq}
            \left(
                \prod_{j \neq i} \frac{1 - \Lambda_i P^{-1}}{1 - \Lambda_i \Lambda_j^{-1}}
            \right)
            =
            \prod_{j \neq i} \frac{H-\lambda_i}{\lambda_j-\lambda_i}.
        \]
        Then, there exists a change of fundamental solution
        $
            P_{q,z}^\textnormal{eq} \in \textnormal{GL}_{N+1} \left(
                \mathcal{M} \left( \mathbb{E}_q \right)
            \right),
        $
        such that the fundamental solution $J^{K\textnormal{th},\textnormal{eq}}$ verifies
        \[
            \gamma_\textnormal{eq} \left(
                \lim_{t \to 0}
                P_{q^t,z}^\textnormal{eq} \cdot \left(
                    \varphi^*_{q^t,z}
                    J^{K\textnormal{th},\textnormal{eq}}\left(q^t,Q\right)
                \right)
            \right)
            =
            J^{\textnormal{coh,eq}}(z,Q).
        \]
    \end{enumerate}
\end{thm}

The proof of this theorem consists of three computations: we begin by studying the confluence of the $q$-difference equation, then of the solution. Then, we compare the limit of the solution to the cohomological $J$-function.
After these three computations, we will give a proof of this theorem.

\subsubsection*{Confluence of the $q$-difference equation.}\label{qkqde:ssection_confluence_qde}

\begin{prop}\label{qkqde:prop_confluence_equation_JK_equivariant}
    Consider the $q$-difference equation (\ref{qkqde:eqn_JK_eq}) satisfied by the $K$-theoretical $J$-function:
    \[
        \left[ \left( 1 - \Lambda_0 \qdeop{Q} \right) \cdots \left( 1 - \Lambda_N \qdeop{Q} \right) - Q \right]  J^{K\textnormal{th}}(q,Q) = 0.
    \]
    Let $\varphi_{q,z}$ be the map
    \[
        \functiondesc{\varphi_{q,z}}{\mathbb{C}}{\mathbb{C}}{Q}{\left( \frac{z}{1-q} \right)^{N+1} Q.}
    \]
    Then, the $q$-pullback of the $q$-difference equation (\ref{qkqde:eqn_JK_eq}) by the isomorphism $\varphi_{q,z}$ is confluent, and its formal limit is the differential equation satisfied by the small equivariant cohomological $J$-function (\ref{qkqde:eqn_pde_jh_eq}). 
\end{prop}
 
\begin{proof}
    Denote by $\delta_q$ the $q$-difference operator $\frac{\qdeop{Q}-\textnormal{Id}}{q-1}$.
    We rewrite the $q$-difference equation (\ref{qkqde:eqn_JK_eq}) to express it with the operators $\delta_q$ instead.
    Using \[\qdeop{Q}=\textnormal{Id}+(q-1)\delta_q,\] we obtain that
    $
    \Delta(q,Q,\delta_q)
    J^{K\textnormal{th},\textnormal{eq}}\left(q^t,Q\right)
    =
    0,
    $
    where $\Delta(q,Q,\delta_q)$ is the $q$-difference operator given by
    \begin{align*}
        &\Delta(q,Q,\delta_q)
        \\
        &=
        \left[
            -Q +
            (1-q)^{N+1} \sum_{i=0}^{N+1}
            \delta_q^i (-1)^i
            \sum_{0 \leq j_1 < \cdots < j_i < N}
            \Lambda_{j_1} \cdots \Lambda_{j_i}
            \prod_{k \in \{0, \dots, N\} - \{j_1, \cdots, j_i\}}
            \frac{1-\Lambda_k}{1-q}
        \right].
    \end{align*}
    As it is written, the formal limit when $q \to 1$ of this operator is given by $-Q$ and thus does not define a differential equation.
    Introduce the $q$-pullback
    \[
        \functiondesc{\varphi_{q,z}}{\mathbb{C}}{\mathbb{C}}{Q}{\left( \frac{z}{1-q} \right)^{N+1} Q.}
    \]
    The $q$-pullback by $\varphi_{q,z}$ of the above $q$-difference equation is given by
    \begin{equation}\label{qkqde:eqn_pullback_JK_eq}
        \left[
        -Q +
        z^{N+1}\sum_{i=0}^{N+1}
        \delta_q^i (-1)^i
        \sum_{0 \leq j_1 < \cdots < j_i < N}
        \!
        \Lambda_{j_1} \cdots \Lambda_{j_i}
        \prod_{k \in \{0, \dots, N\} - \{j_1, \cdots, j_i\}}
        \frac{1-\Lambda_k}{1-q}
        \right]
        \!
        f(q,Q)
        = 0.
    \end{equation}
    Since the relation $\Lambda_i=q^\frac{-\lambda_i}{z}$ holds for all $i \in \{0, \dots, N\}$, this $q$-difference equation is confluent.
    Using the same relation again, we can compute its formal limit when $q \to 1$.
    The resulting formal limit coincides with the developed expression of the differential equation (\ref{qkqde:eqn_pde_jh_eq}) satisfied by the cohomological $J$-function.
\end{proof}

\begin{remark}
    The $q$-pullback $\varphi_{q,z}$ defined in Proposition \ref{qkqde:prop_confluence_equation_JK_equivariant} is the only $q$-pullback of the form $Q \mapsto \left( \frac{z}{1-q} \right)^\lambda Q$, with $\lambda \in \mathbb{Z}$, which defines a confluent $q$-difference system whose formal limit is non zero.
\end{remark}
 
\subsubsection*{Confluence of the solution.}\label{qkqde:ssection_confluence_sol}

    The $q$-difference system associated to the $q$-pullbacked equation (\ref{qkqde:eqn_pullback_JK_eq}) has a fundamental solution obtained from the $J$-function $J^{K\textnormal{th,eq}}(q,Q)$, which is explicitly given by 
    \begin{equation}\label{qkqde:eqn_fundamental_solution_before_confluence_equivariant}
        \mathcal{X}^{K\textnormal{th,eq}}\left(q,Q\right)
        =
        \begin{pmatrix}
            J^{K\textnormal{th},\textnormal{eq}}_{|P=\Lambda_0}\left(q,\left( \frac{1-q}{z} \right)^{N+1} Q\right) &   \cdots  &   J^{K\textnormal{th},\textnormal{eq}}_{|P=\Lambda_N}\left(q,\left( \frac{1-q}{z} \right)^{N+1} Q\right)    \\
            \vdots                                                                                      &   \ddots  &   \vdots                                                      \\
            \delta_q^N J^{K\textnormal{th},\textnormal{eq}}_{|P=\Lambda_0}\left(q,\left( \frac{1-q}{z} \right)^{N+1} Q\right) &   \cdots  &  \delta_q^N  J^{K\textnormal{th},\textnormal{eq}}_{|P=\Lambda_N}\left(q,\left( \frac{1-q}{z} \right)^{N+1} Q\right)
        \end{pmatrix}.
    \end{equation}
    The condition $\Lambda_i \Lambda_j^{-1} \notin q^\mathbb{Z}$ for $i \neq j$ implies that this matrix is invertible.

\begin{prop}\label{qkqde:prop_jk_eq_confl_sol}
    There exists a change of fundamental solution, denoted by $P_{q,z}^\textnormal{eq} \in \textnormal{GL}_{N+1}\left( \mathcal{M}\left( \mathbb{E}_q \right) \right)$ such that the new fundamental solution $\mathcal{X}^{K\textnormal{th,eq}}\left(q,Q\right)P_{q,z}^\textnormal{eq}$ obtained from Equation $(\ref{qkqde:eqn_fundamental_solution_before_confluence_equivariant})$ is given by
    \[
        \left( 
            \mathcal{X}^{K\textnormal{th,eq}}\left(q,Q\right)P_{q,z}^\textnormal{eq} 
        \right)_{li}
        =
        \left( \delta_q \right)^l
        \Lambda_i^{-\ell_q(Q)}
        \sum_{d \geq 0} 
        \frac{1}{z^{d(N+1)}}
        \frac{(1-q)^{d(N+1)}Q^d}{\left(q\Lambda_0 \Lambda_i^{-1}, \dots ,q, \dots ,q\Lambda_N \Lambda_i^{-1} ;q \right)_d}
    \]
    Moreover, this fundamental solution is confluent.
\end{prop}

\begin{proof}
    We begin by trying to compute the limit of the fundamental solution $\mathcal{X}^{K\textnormal{th,eq}}\left(q,Q\right)$
    when $q$ tends to 1.
    Let $i \in \{0, \dots, N\}$. We have
    \begin{align*}
        &J^{K\textnormal{th},\textnormal{eq}}_{|P=\Lambda_i}\left(q,\left( \frac{1-q}{z} \right)^{N+1} Q\right)
        \\
        &= \Lambda_i^{-\ell_q \left( \left(\frac{q-1}{z}\right)^{N+1} Q \right)}
        \sum_{d \geq 0} \frac{(1-q)^{d(N+1)}}{z^{d(N+1)}(q \Lambda_0 \Lambda_i^{-1};q)_d \cdots (q \Lambda_N \Lambda_i^{-1};q)_d)} Q^d.
    \end{align*}
    First let us check that every term in the sum indexed by $d$ has a well defined limit when $q$ tends to 1: the relation $\Lambda_i = q^{-\lambda_i/z}$ gives that for any $r \in \mathbb{Z}$,
    \[
        \lim_{q \to 1} \frac{1-q}{1-q^r\Lambda_j \Lambda_i^{-1}}
        =
        \frac{z}{r+\lambda_i-\lambda_j}.
    \]
    Therefore, we have
    \[
        \lim_{q \to 1} \frac{(1-q)^{d(N+1)}}{z^{d(N+1)}(\Lambda_0 \Lambda_i^{-1};q)_d \cdots (\Lambda_N \Lambda_i^{-1};q)_d} Q^d
        =
        Q^d \prod_{r=1}^d \prod_{j=0}^N \frac{1}{(\lambda_i-\lambda_j+rz)}.
    \]

    It remains to deal with the divergent coefficient
    $
        \Lambda_i^{-\ell_q \left( \left(\frac{q-1}{z}\right)^{N+1} Q \right)}.
    $
    Notice that the two functions given by $\ell_q(Q)$ and $\ell_q \left( \left(\frac{q-1}{z}\right)^{N+1} Q \right)$ are both $q$-logarithms, i.e. solutions of the $q$-difference equation $\qdeop{Q}f_q(Q) = f_q(Q)+1$.
    Therefore, there exists a change of fundamental solution
    $P_{q,z}^\textnormal{eq} \in  \textnormal{GL}_{N+1}\left( \mathcal{M}\left( \mathbb{E}_q \right) \right)$
    which allows us, in the formula of the fundamental solution $\mathcal{X}^{K\textnormal{th,eq}}\left(q,Q\right)$, to change the divergent $q$-logarithms $\ell_q \left( \left(\frac{q-1}{z}\right)^{N+1} Q \right)$ into the convergent $q$-logarithms $\ell_q(Q)$.
    Then, by Proposition \ref{qde:prop_specfns_limits},
    \[
        \lim_{t \to 0} \Lambda_i^{-\ell_{q^t}(Q)} 
        =
        \lim_{t \to 0}e^{\frac{\lambda_i}{z} \log(q^t) \ell_{q^t}(Q)}
        =
        Q^\frac{\lambda_i}{z}.
    \]
    Therefore, the transformed fundamental solution $\mathcal{X}^{K\textnormal{th,eq}}\left(q,Q\right)P_{q,z}^\textnormal{eq}$ is confluent, and its coefficients are given by
    \[
        \left( 
            \mathcal{X}^{K\textnormal{th,eq}}\left(q,Q\right)P_{q,z}^\textnormal{eq} 
        \right)_{li}
        =
        \left( \delta_q \right)^l
        \Lambda_i^{-\ell_q(Q)}
        \sum_{d \geq 0} 
        \frac{1}{z^{d(N+1)}}
        \frac{(1-q)^{d(N+1)}Q^d}{\left(q\Lambda_0 \Lambda_i^{-1}, \dots q, \dots ,q\Lambda_N \Lambda_i^{-1} ;q \right)_d}.
    \]
\end{proof}

\subsubsection*{Comparison between confluence of quantum $K$-theory and quantum cohomology.}\label{qkqde:ssubsection_proofconclusion}

Recall that we use a basis of the equivariant $K$-theory given by $\eta_i = \prod_{j \neq i} \frac{1 - \Lambda_j P^{-1}}{1 - \Lambda_j \Lambda_i^{-1}} \in K_{T^{N+1}} \left( \mathbb{P}^N \right)$.

\begin{defin}\label{qkqde:def_confluence_jk_equivariant}
    Denote by $P_{q,z}^\textnormal{eq} \cdot \varphi_{q,z}^* J^{K\textnormal{th,eq}}$ the $K$-theoretical function obtained from the first row of the transformed fundamental solution:
    \[
        P_{q,z}^\textnormal{eq} \cdot \varphi_{q,z}^* J^{K\textnormal{th,eq}}
        =
        \sum_{i=0}^N
        \left(\mathcal{X}^{K\textnormal{th,eq}}\left(q,Q\right)P_{q,z}^\textnormal{eq}\right)_{0i} \eta_i.
    \]
    By Proposition \ref{qkqde:prop_jk_eq_confl_sol}, the limit when $q^t$ tends to 1 of the function $P_{q,z}^\textnormal{eq} \cdot \varphi_{q,z}^* J^{K\textnormal{th,eq}}$ is well defined. We define the $K$-theoretical function $\textnormal{confluence}\left(J^{K\textnormal{th},\textnormal{eq}}\right)$ by
        \[
            \textnormal{confluence}\left(J^{K\textnormal{th},\textnormal{eq}}\right)(z,Q)
            =
            \lim_{t \to 0}
            P_{q^t,z}^\textnormal{eq} \cdot \varphi_{q^t,z}^* J^{K\textnormal{th,eq}}(q^t,Q).
        \]
\end{defin}

\begin{prop}\label{qkqde:prop_confluence_JK_equals_JH_equivariant}
    Consider the isomorphism of rings $\gamma_\textnormal{eq} : K_{T^{N+1}}\left( \mathbb{P}^N \right) \to H_{T^{N+1}}^* \left( \mathbb{P}^N ;\mathbb{Q} \right)$ given by $\gamma_\textnormal{eq}(\eta_i)=\prod_{j \neq i} \frac{H-\lambda_i}{\lambda_j-\lambda_i}$ for all $i \in \{0, \dots, N\}$
    Then,
    \[
        \gamma_\textnormal{eq}\left( \textnormal{confluence}\left(J^{K\textnormal{th},\textnormal{eq}}\right)(z,Q) \right)
        =
        J^{\textnormal{coh},\textnormal{eq}}(z,Q).
    \]
\end{prop}

\begin{proof}
    We have
    \begin{align*}
        &P_{q,z}^\textnormal{eq} \cdot \varphi_{q,z}^* J^{K\textnormal{th,eq}}(t,z,Q)
        \\
        &=
        \sum_{i=0}^N
        \left(
            \Lambda_i^{-\ell_q(Q)}
            \sum_{d \geq 0} 
            \frac{1}{z^{d(N+1)}}
            \frac{(1-q)^{d(N+1)}Q^d}{\left(q\Lambda_0 \Lambda_i^{-1}, \dots q, \dots ,q\Lambda_N \Lambda_i^{-1} ;q \right)_d}
        \right)
        \eta_i.
    \end{align*}
    Thus,
    \[
        \textnormal{confluence}\left(J^{K\textnormal{th},\textnormal{eq}}\right)(z,Q)
        =
        \sum_{i=0}^N
        \left(
            Q^\frac{\lambda_i}{z}
            \sum_{d \geq 0} Q^d
            \prod_{r=1}^d \prod_{j=0}^N \frac{1}{(\lambda_i-\lambda_j+rz)}
        \right)
        \eta_i.
    \]
    We conclude using $\gamma_\textnormal{eq}(\eta_i) = \prod_{j \neq i} \frac{H-\lambda_i}{\lambda_j-\lambda_i}$, recalling that
    \[
        J^{\textnormal{coh},\textnormal{eq}}_{|H=\lambda_i}(z,Q)
        =
        Q^\frac{\lambda_i}{z}
        \sum_{d \geq 0} Q^d
        \prod_{r=1}^d \prod_{j=0}^N \frac{1}{(\lambda_i-\lambda_j+rz)}.
    \]
\end{proof}

\subsubsection*{Summary of the previous results.}\label{qkqde:ssubsection_proofsummary}

We have now all the ingredients to give the proof of Theorem \ref{qkqde:thm_confluence_jk_eq}.

\begin{proof}[Proof of Theorem \ref{qkqde:thm_confluence_jk_eq}]
    \textbf{Confluence of the equation}. Using the $q$-pullback $\varphi_{q,z}$ of Proposition \ref{qkqde:prop_confluence_equation_JK_equivariant}, we obtain a confluent $q$-difference system. Its limit is the differential equation associated to the small equivariant cohomological $J$-function.
    
    \textbf{Confluence of the solution}. As done in Equation \eqref{qkqde:eqn_fundamental_solution_before_confluence_equivariant}, we can encode the equivariant $K$-theoretical $J$-function as a fundamental solution of the $q$-pullback of the system (\ref{qkqde:eqn_pullback_JK_eq}), which we denote by $\mathcal{X}^{K\textnormal{th,eq}}\left(q,Q\right)$ in Equation \ref{qkqde:eqn_fundamental_solution_before_confluence_equivariant}.
    By Proposition \ref{qkqde:prop_jk_eq_confl_sol}, there exists a $q$-constant transformation $P_{q,z}^\textnormal{eq} \in  \textnormal{GL}_{N+1}\left( \mathcal{M}\left( \mathbb{E}_q \right) \right)$ such that the fundamental solution $\mathcal{X}^{K\textnormal{th,eq}}\left(q,Q\right) P_{q,z}^\textnormal{eq}$ is confluent.
    
    \textbf{Comparison with quantum cohomology}.
    The first row of the fundamental solution after the $q$-constant transformation $P_{q,z}$, $\mathcal{X}^{K\textnormal{th,eq}}\left(q,Q\right) P_{q,z}^\textnormal{eq}$, defines another $K$-theoretical function, which we denote by $P_{q,z}^\textnormal{eq} \cdot \varphi_{q,z}^*J^{K\textnormal{th,eq}}$ in Definition \ref{qkqde:def_confluence_jk_equivariant}.
    Since the fundamental solution was confluent, this function has a well defined limit when $q^t \to 1$.
    Using Proposition \ref{qkqde:prop_confluence_JK_equals_JH_equivariant}, we have
    \[
        \gamma_\textnormal{eq} \left( \lim_{t \to 0} P_{q^t,z}^\textnormal{eq} \cdot \varphi_{q^t,z}^*J^{K\textnormal{th,eq}}(q^t,Q) \right)
        =
        J^{\textnormal{coh,eq}}(z,Q).
    \]
\end{proof}

\subsection{Confluence and non equivariant limit}

Since Givental's equivariant $J$-functions have well defined non equivariant limit by setting $\lambda_i \to 0$ and $\Lambda_i \to 1$ for all $i \in \{0, \dots, N\}$, one may wonder if there is a statement analogue to the Theorem \ref{qkqde:thm_confluence_jk_eq} for non equivariant $J$-function.
While the answer is positive, the details are slightly more technical.

\subsubsection*{Definitions and statement of the theorem.}

\begin{remark}
    A basis of the non equivariant $K$-theory 
    \[
    K\left( \mathbb{P}^N \right) \simeq \mathbb{Z}[P,P^{-1}] \left/\left( \left(1-P^{-1} \right)^{N+1} \right) \right.
    \]
    is given by the integer powers of $1-P^{-1}$.
    Notice that the the non equivariant limit of the equivariant basis given by $\eta_i = \prod_{j \neq i} \frac{1 - \Lambda_j P^{-1}}{1 - \Lambda_j \Lambda_i^{-1}} \in K_{T^{N+1}} \left( \mathbb{P}^N \right)$ is not a basis the non equivariant $K$-theory.
\end{remark}

\begin{defin}
    Let $X = \mathbb{P}^N$ and let $P=\mathcal{O}(1) \in K\left( \mathbb{P}^N \right)$ be the anti-tautological bundle.
    \textit{Givental's small $K$-theoretical $J$-function} is the function given by
    \[
        J^{K\textnormal{th}}(q,Q) = P^{-\ell_q(Q)}\sum_{d \geq 0} \frac{Q^d}{\left(qP^{-1};q\right)_d^{N+1}}
        \in
        K \left( \mathbb{P}^N \right) \otimes \mathbb{C}(q) [\![Q]\!],
    \]
    where
    \[
        P^{-\ell_q(Q)}
        =
        \left( 1 - \left(1-P^{-1}\right) \right)^{\ell_q(Q)}
        =
        \sum_{k \geq 0} (-1)^k \binom{\ell_q(Q)}{k} \left( 1 - P^{-1} \right)^k,
    \]
    and
    \[
        \binom{\ell_q(Q)}{k} = \frac{1}{k!} \prod_{r=0}^{k-1} (\ell_q(Q)-r).
    \]
\end{defin}

\begin{prop}[\cite{Iri_Mil_Ton_qk}, Equation 10]
    The $J$-function $J^{K\textnormal{th}}(q,Q)$ is a solution of the $q$-difference equation
    \begin{equation}\label{qkqde:eqn_JK}
        \left[ \left( 1 - \qdeop{Q} \right)^{N+1} - Q \right]  \widetilde{J^{K\textnormal{th}}}(q,Q) = 0.
    \end{equation}
    This $q$-difference equation is regular singular at $Q=0$. 
\end{prop}

\begin{defin}[\cite{CK_book}, Proposition 11.2.1; see also Equation (10.39)]
    Givental's small cohomological $J$-function is given by the expression
    \begin{equation*}
        J^\textnormal{coh}(z,Q) = Q^{\frac{H}{z}} \sum_{d \geq 0} \frac{Q^d}{\prod_{r=1}^d \left( H + rz\right)^{N+1}}
        \in
        H\left( \mathbb{P}^N \right) \otimes \mathbb{C}[z,z^{-1}][\![Q]\!].
    \end{equation*}
\end{defin}

\begin{prop}[\cite{CK_book}, Equation 10.38]
    This function is a solution of the differential equation
    \begin{equation}\label{qkqde:eqn_pde_JH}
        \left[(zQ \partial_Q)^{N+1} - Q\right]
        J^\textnormal{coh}(z,Q) = 0.
    \end{equation}
\end{prop}

\begin{thm}\label{qkqde:thm_confluence_jk}
    Let $X = \mathbb{P}^N$.
    Denote by $J^{K\textnormal{th}}$ (resp. $J^{\textnormal{coh}}$) Givental's small $K$-theoretical (resp. cohomological) $J$-function.
    Let $q \in \mathbb{C}, 0<|q|<1$ and $z \in \mathbb{C}^*$.
    The following statements hold:
    \begin{enumerate}[label=(\roman*)]
        \item Consider the application $\varphi_{q,z}$ defined by
        \[
            \functiondesc{\varphi_{q,z}}{\mathbb{C}}{\mathbb{C}}{Q}{\left( \frac{z}{1-q} \right)^{N+1} Q}
        \]
        The pullback by $\varphi_{q,z}$ of the $q$-difference equation (\ref{qkqde:eqn_JK}) satisfied by the $K$-theoretical $J$-function is confluent.
        Moreover, its formal limit when $q \to 1$ is the differential equation (\ref{qkqde:eqn_pde_JH}) satisfied by the cohomological $J$-function.
        
        \item Let $
            \mathbb{E}_q
        $
        be the complex torus 
        $
            \mathbb{C}^* / q^\mathbb{Z}
        $
        and
        $\mathcal{M} \left( \mathbb{E}_q \right)$
        be the space of meromorphic functions on said complex torus. Consider the isomorphism of rings
        $
            \gamma : 
            K \left( \mathbb{P}^N \right) \otimes \mathbb{C} 
            \to 
            H^*\left( \mathbb{P}^N, \mathbb{C} \right)
        $
        defined by, for all $i \in \{0, \dots, N\}$
        \[
            \gamma
            \left(
                \left( 1 - P^{-1} \right)^i
            \right)
            =
            H^i
        \]
        Then, there exists a change of fundamental solution
        $
            P_{q,z} \in \textnormal{GL}_{N+1} \left(
                \mathcal{M} \left( \mathbb{E}_q \right)
            \right)
        $
        such that the fundamental solution $J^{K\textnormal{th}}$ verifies
        \[
            \gamma \left(
                \lim_{t \to 0}
                P_{q^t,z} \cdot \left(
                    \varphi^*_{q^t,z}
                    J^{K\textnormal{th}}\left(q^t,Q\right)
                \right)
            \right)
            =
            J^{\textnormal{coh}}(z,Q)
        \]
    \end{enumerate}
\end{thm}

The plan of the proof is the same as in the equivariant setting (Theorem \ref{qkqde:thm_confluence_jk_eq}): first we study the confluence of the $q$-difference equation, then the confluence of Givental's $J$-function as a fundamental solution, which we compare to the cohomological $J$-function.
However, the confluence of the fundamental solution requires a different change of fundamental solution 
$
    P_{q,z} \in \textnormal{GL}_{N+1} \left(
    \mathcal{M} \left( \mathbb{E}_q \right)
    \right)
$, which is slightly more complex than in the equivariant case.
For a detailed proof of this statement, we will refer to Section VI.2 of \cite{roquefeuil_thesis}.

\subsubsection*{Confluence of the $q$-difference equation.}

\begin{prop}\label{qkqde:prop_jk_eqn_pullback}
    Consider the $q$-difference equation (\ref{qkqde:eqn_JK}) :
    \[
        \left( 1 - q^{Q \partial_Q} \right)^{N+1} f(q,Q) = Q f(q,Q).
    \]
    Let $z \in \mathbb{C}^*$ and let $\varphi_{q,z}$ be the function
    \[
        \functiondesc{\varphi_{q,z}}{\mathbb{C}}{\mathbb{C}}{Q}{\left( \frac{z}{1-q} \right)^{N+1} Q}.
    \]
    
    Then, the $q$-pullback of the $q$-difference equation (\ref{qkqde:eqn_JK}) by $\varphi_{q,z}$ is confluent, and its limit is the differential equation (\ref{qkqde:eqn_pde_JH}) satisfied by Givental's small cohomological $J$-function.
\end{prop}

The proof of this proposition can be obtained by setting $\Lambda_i \to 1, \lambda_i \to 0$ for all $i \in \{0, \dots, N\}$ in the proof of Proposition \ref{qkqde:prop_confluence_equation_JK_equivariant}
Writing $\delta_q = \frac{\qdeop{Q}-\textnormal{id}}{q-1}$, the pullback by $\varphi_{q,z}$ of the $q$-difference equation (\ref{qkqde:eqn_JK}) is given by
\begin{equation}\label{qkqde:eqn_qde_jk_pullback}
    \left[
        \left( z \delta_q \right)^{N+1} - Q
    \right]
    J^{K\textnormal{th}}\left(q,\varphi_{q,z}^{-1}(Q)\right)
    =
    0.
\end{equation}

\subsubsection*{Confluence of the fundamental solution.}

    Consider the decomposition
    \[
        J^{K\textnormal{th}}(q,Q)
        =
        \sum_{i=0}^N J_i(q,Q) \left(1-P^{-1} \right)^i
        \in
        K \left( \mathbb{P}^N \right) \otimes \mathbb{C}(q) [\![Q]\!]
    \]
    Givental's small $K$-theoretical $J$-functions can be encoded in the fundamental solution of the $q$-difference equation (\ref{qkqde:eqn_qde_jk_pullback}) given by the matrix

    \begin{align}\label{qkqde:eqn_pullback_of_fundamental_solution_matrix_form}
    &\mathcal{X}^{K\textnormal{th}} \left(q,Q\right) \nonumber
    \\
    &=
    \begin{pmatrix}
        J_0\left(q,\left( \frac{1-q}{z}\right)^{N+1} Q \right)            &   J_1\left(q,\left( \frac{1-q}{z}\right)^{N+1} Q \right)            &    \cdots  &   J_N\left(q,\left( \frac{1-q}{z}\right)^{N+1} Q \right)               \\
        \delta_q J_0\left(q,\left( \frac{1-q}{z}\right)^{N+1} Q \right)   &   \delta_q  J_1\left(q,\left( \frac{1-q}{z}\right)^{N+1} Q \right)  &    \cdots  &   \delta_q J_N\left(q,\left( \frac{1-q}{z}\right)^{N+1} Q \right)      \\
        \vdots                          &   \vdots                          &    \ddots  &   \vdots                             \\
        \delta_q^N J_0\left(q,\left( \frac{1-q}{z}\right)^{N+1} Q \right) &   \delta_q^N J_1\left(q,\left( \frac{1-q}{z}\right)^{N+1} Q \right) &    \cdots  &   \delta_q^N J_N\left(q,\left( \frac{1-q}{z}\right)^{N+1} Q \right)    \\
    \end{pmatrix}
\end{align}

\begin{prop}[\cite{roquefeuil_thesis}, Proposition VI.2.3.3]\label{qkqde:prop_JK_confluence_sol}
    There exists an explicit $q$-constant matrix $P_{q,z} \in \textnormal{GL}_{N+1}\left( \mathcal{M}\left( \mathbb{E}_q \right) \right)$ such that the new fundamental solution $ \mathcal{X}^{K\textnormal{th}} \left(q,Q \right) P_{q,z}$ obtained from Equation \ref{qkqde:eqn_pullback_of_fundamental_solution_matrix_form} is given by
    \[
        \left(
            \mathcal{X}^{K\textnormal{th}}(q,Q) \! P_{q,z}
        \right)_{li}
        =
        \left(\delta_q\right)^l
        \!
        \sum_{\substack{
                0 \leq a, b \leq N  \\  a + b = i
            }}
        \left( \frac{q-1}{z} \right)^a 
        \binom{\ell_{q} \left( Q \right)}{a}
        \left( \frac{1-q}{z} \right)^b f_b \left(
            q,\left( \frac{1-q}{z}\right)^{N+1} Q
        \right),
    \]
    where the functions $f_b$ are defined by
    \begin{align*}
        &f_b(q,Q) = \sum_{d \geq 0} \frac{Q^d}{(q;q)_d^{N+1}}
        \times
        \\
        &\times
        \left(
            \sum_{k=0}^N \sum_{\substack{
            0 \leq j_1,\dots,j_N \leq N \\  j_1 + \cdots + j_N = k \\ j_1 + 2 j_2 + \cdots + N j_N = b
            }}
            (-1)^k \frac{(N+k)!}{N! j_1! \cdots j_N!}
            \prod_{l=1}^N \left(
                \sum_{1 \leq m_1 < \cdots < m_l \leq d} \frac{q^{m_1+\cdots+m_l}}{(1-q^{m_1}) \cdots (1-q^{m_l})}
            \right)^{j_l}
        \right).
    \end{align*}
    Moreover, this fundamental solution has a non trivial limit when $q^t$ tends to $1$.
    \qed
\end{prop}
In that case, we can not use the proof of the equivariant statement (Proposition \ref{qkqde:prop_jk_eq_confl_sol}), as the non equivariant limit of our basis in equivariant $K$-theory is not a basis in non equivariant $K$-theory.
However, the technique will be similar, so we will refer to Proposition VI.2.3.3 of \cite{roquefeuil_thesis} for the complete proof.
Just like in Proposition \ref{qkqde:prop_jk_eq_confl_sol}, we will need to change the $q$-logarithms, but to obtain a well defined limit, the $i$-th column of the  matrix \eqref{qkqde:eqn_pullback_of_fundamental_solution_matrix_form} has to be multiplied by the factor $\left( \frac{q-1}{z} \right)^i$.

\subsubsection*{Comparison with quantum cohomology.}

To complete the proof of Theorem \ref{qkqde:thm_confluence_jk}, it remains to compare the limit of the first row of the fundamental solution $\mathcal{X}^{K\textnormal{th}} \left(q,Q \right)P_{q,z}$ with Givental's small cohomological $J$-function.

\begin{defin}
    We denote by $P_{q,z} \cdot \varphi_{q,z}^* J^{K\textnormal{th}}(q,Q)$ the $K$-theoretical function defined by
    \[
        P_{q,z} \cdot \varphi_{q,z}^* J^{K\textnormal{th}}(q,Q)
        =
        \sum_{i=0}^N
        \left(
            \mathcal{X}^{K\textnormal{th}} \left(q,Q \right)P_{q,z}
        \right)_{0i}
        H^i
        \in
        K\left( \mathbb{P}^N \right) \otimes \mathbb{C}(q,z)[\![Q]\!]
    \]
\end{defin}

\begin{prop}
    Consider the ring automorphism
    $ \gamma : K\left( \mathbb{P}^N \right)_\mathbb{Q} \to H^*\left( \mathbb{P}^N; \mathbb{Q}\right)$ defined by $\gamma(1-P^{-1})=H$.
    The following asymptotic holds
    \[
        \gamma \left( \lim_{t \to 0} P_{q^t,z} \cdot \varphi_{q^t,z}^* J^{K\textnormal{th}}(q^t,Q) \right)
        =
        J^{\textnormal{coh}}(z,Q).
    \]
\end{prop}

\begin{proof}
    Using the characterisation of the change of fundamental solution of Proposition \ref{qkqde:prop_JK_confluence_sol}, we have to compute the limits of the terms for all $i \in \{0, \dots, N\}$
    \[
        \sum_{\substack{
                0 \leq a, b \leq N  \\  a + b = i
            }}
        \left( \frac{q-1}{z} \right)^a 
        \binom{\ell_{q} \left( Q \right)}{a}
        \left( \frac{1-q}{z} \right)^b f_b \left(
            q,\left( \frac{1-q}{z}\right)^{N+1} Q
        \right)
    \]
    The decomposition of the cohomological $J$-function in the basis $(1,H,\dots,H^N)$ is given by
    \begin{equation}\label{qkqde:eqn_jh_decomposition}
        J^\textnormal{coh}(z,Q) =
        \sum_{i=0}^N H^i
        \left[
            \sum_{\substack{
                0 \leq a, b \leq N  \\  a + b = i
            }}
            \frac{1}{a !} \left( \frac{\log(Q)}{z} \right)^a
            g_b(z,Q)
        \right],
    \end{equation}
    where
    \begin{align*}
        g_b(z,Q)
        &=
        \sum_{d \geq 0} \frac{Q^d}{\left( z^d d!\right)^{N+1}}
        \times
        \\
        &\times
        \left(
            \frac{1}{z^b}
                \sum_{k=0}^N \sum_{\substack{
            0 \leq j_1,\dots,j_N \leq N \\  j_1 + \cdots + j_N = k \\ j_1 + 2 j_2 + \cdots + N j_N = b
            }}
            (-1)^k \frac{(N+k)!}{N! j_1! \cdots j_N!}
                \prod_{l=1}^N \left(
                \sum_{1 \leq m_1 < \cdots < m_l \leq d} \frac{1}{m_1 \cdots m_l)}
            \right)^{j_l}
        \right).
    \end{align*}
    We observe that
    \[
        \lim_{t \to 0}
        \left( \frac{1-q^t}{z} \right)^b f_b \left(
            q^t,\left( \frac{1-q^t}{z}\right)^{N+1} Q
        \right)
        =
        g_b(z,Q).
    \]
    Using Proposition \ref{qde:prop_specfns_limits}, we also have that
    \[
        \lim_{t \to 0}
        \left( \frac{q^t-1}{z} \right)^a 
        \binom{\ell_{q^t} \left( Q \right)}{a}
        =
        \frac{1}{a !} \left( \frac{\log(Q)}{z} \right)^a.
    \]
    Using these two limits, we obtain that
    \[
        \lim_{t \to 0} P_{q^t,z} \cdot \varphi_{q^t,z}^* J^{K\textnormal{th}}(q^t,Q)
        =
        \sum_{i=0}^N \left( 1-P^{-1}\right)^i
        \left[
            \sum_{\substack{
                0 \leq a, b \leq N  \\  a + b = i
            }}
            \frac{1}{a !} \left( \frac{\log(Q)}{z} \right)^a
            g_b(z,Q)
        \right].
    \]
    Applying $\gamma$ and comparing with Equation (\ref{qkqde:eqn_jh_decomposition}), we find the desired result.
\end{proof}
\section{$q$-monodromy of the $q$-difference equation of projective spaces}\label{section:qmonodromy}

The goal of this section is to compute some monodromy data for the $q$-difference equation satisfied by the small $K$-theoretical $J$-function of projective spaces.
This monodromy data takes the form of a connection matrix, computing the base change from the $J$-function (fundamental solution at $Q=0$) to a fundamental solution at $Q=\infty$.

Before starting, let us mention some references on $q$-monodromy. A treatment of the regular singular case can be found in \cite{Har_Sau_Sin_book}, the end result being Theorem 3.4.9 p.134. For some irregular case, we can refer to \cite{Dreyfus_confluence_sing_irreg} in general, \cite{Adachi_borel_sum} for irregular (unilateral) $q$-hypergeometric series, and \cite{Wen:qkqde_quinticthreefold} for the $q$-difference equation associated to Fermat's quintic threefold.

\subsection{Fundamental solution at $\infty$}

In this Subsection, we begin by constructing a fundamental solution of the $q$-difference equation \eqref{qkqde:eqn_JK_eq} at $Q=\infty$.
This solution is built by looking for a formal solution to the $q$-difference equation, then using a $q$-analogue of the Borel--Laplace transform to obtain an analytic solution.

\subsubsection{Formal solution}

We denote by $w=Q^{-1}$ our coordinate at $Q=\infty$.
Notice that if $f,g$ are two complex functions so that $g(w) = f(1/w) = f(Q)$, then $\qdeop{Q} f = \left( \qdeop{w} \right)^{-1} g$.
Therefore, the $q$-difference equation $(\Delta_q)$ in the new local coordinate $w$ becomes
\begin{equation}\label{stokes:eqn_qde_infinity}
	\left[
		(-1)^{N+1} q^{N+1} \Lambda_0 \cdots \Lambda_N w
		\left( 1 - \Lambda_0^{-1} \qdeop{w} \right) \cdots \left( 1 - \Lambda_N^{-1} \qdeop{w} \right)
		- \left( \qdeop{w} \right)^{N+1}
	\right]
	g_q(w) = 0.
\end{equation}

\begin{notation}
	We recall some $q$-analogues of hypergeometric functions.
	Let $r,s \in \mathbb{Z}_{\geq 0}$ and $a_1, \dots, a_r, b_1, \dots, b_s \in \mathbb{C}$. The notation for a (unilateral) $q$-hypergeometric series is
	\[
        \tensor[_{r}]{\varphi}{_s} \left(
            \begin{gathered}
                a_1 \quad \cdots \quad a_r \\
                b_1 \quad \cdots \quad b_s
            \end{gathered}
            \middle| q, z
        \right)
        =
        \sum_{d \geq 0}
        \frac{\left(a_1,\dots,a_r ;q\right)_d}{\left(q,b_1,\dots,b_s ;q\right)_d}
        \left(
            (-1)^d q^{\frac{d(d-1)}{2}}
        \right)^{1+s-r}
        z^d.
	\]
	Keeping the same notations, we define bilateral $q$-hypergeometric series by
	\[
		\tensor[_{r}]{\psi}{_s} \left(
            \begin{gathered}
                a_1 \quad \cdots \quad a_r \\
                b_1 \quad \cdots \quad b_s
            \end{gathered}
            \middle| q, w
        \right)
        =
        \sum_{d \in \mathbb{Z}}
        \frac{(a_1, \dots, a_r;q)_d}{(b_1, \dots, b_s;q)_d}
        \left(
        	(-1)^d q^\frac{d(d-1)}{2}
        \right)^{s-r}
        w^d.
	\]
\end{notation}

Let $\alpha \in \mathbb{C}^*$. We will use the following ansatz ton construct our fundamental solution at $w=0$.

\begin{lemma}\label{stokes:lemma_formal_fund_sol_ansatz_qde}
	Let $h_q$ be a complex function and set $g_q(w) := e_{q,\alpha^{-1} P}(w) h_q(w)$.
	The function $g_q$ is a solution of the $q$-difference equation (\ref{stokes:eqn_qde_infinity}) if and only if the function $h_q$ is a solution of the following $q$-difference equation:
	\begin{equation}\label{stokes:eqn_qde_laurent_series}
		\left[
			(-1)^{N+1}
			q^{N+1} w \Lambda_0 \cdots \Lambda_N
			\prod_{i=0}^N \left(
				1 - \Lambda_i^{-1} \alpha^{-1} P \qdeop{w}
			\right)
			-
			\left(
				\alpha^{-1} P
			\right)^{N+1}
			\left(
				\qdeop{w}
			\right)^{N+1}
		\right]
		h_q(w)
		=
		0
	\end{equation} 
\end{lemma}

\begin{proof}
	Assume the function $g_q$ is a solution of the $q$-difference equation (\ref{stokes:eqn_qde_infinity}).
	The functions $g_q, h_q$ are related by the relation $g_q(w) = e_{\alpha^{-1} P}(w) h_q(w)$. Therefore,
	\[
		\qdeop{w} g_q(w)
		=
		e_{\alpha^{-1} P}(w) \alpha^{-1} P \qdeop{w} h_q(w).
	\]
	Thus, when $q$-difference operator in \eqref{stokes:eqn_qde_infinity} to $g_q$, we obtain
	\begin{align*}
		&\left[
			(-1)^{N+1} q^{N+1} \Lambda_0 \cdots \Lambda_N w
			\left( 1 - \Lambda_0^{-1} \qdeop{w} \right) \cdots \left( 1 - \Lambda_N^{-1} \qdeop{w} \right)
			- \left( \qdeop{w} \right)^{N+1}
		\right]
    	g_q(w)
    	\\
    	=
		&e_{\alpha^{-1} P}(w)
		\left[
			(-1)^{N+1}
			q^{N+1} w \Lambda_0 \cdots \Lambda_N
			\prod_{i=0}^N \left(
				1 - \Lambda_i^{-1} \alpha^{-1} P \qdeop{w}
			\right)
			-
			\left(
				\alpha^{-1} P
			\right)^{N+1}
			\left(
				\qdeop{w}
			\right)^{N+1}
		\right]
		h_q(w),
	\end{align*}
	which is zero by assumption that $g_q$ is a solution of the $q$-difference equation (\ref{stokes:eqn_qde_infinity}).
\end{proof}

\begin{lemma}\label{stokes:lemma_ansatz_qde_solution}
	The $q$-difference equation (\ref{stokes:eqn_qde_laurent_series}) of the previous lemma admits as a solution the following formal Laurent series
	\begin{equation}\label{stokes:eqn_laurent_series_solution}
		h_q(w)
		=
		\tensor[_{N+1}]{\psi}{_0} \left(
            \begin{gathered}
                \Lambda_0^{-1} \alpha^{-1} P \quad \cdots \quad \Lambda_N^{-1} \alpha^{-1} P \\
                -
            \end{gathered}
            \quad \middle| \quad
            q, \left( \alpha P^{-1} \right)^{N+1} \Lambda_0 \cdots \Lambda_N w
        \right)
	\end{equation}
\end{lemma}

\begin{remark}
	Since $|q|<1$, applying the ratio test to the positive part of the Laurent series (\ref{stokes:eqn_laurent_series_solution}) shows that its convergence ray is 0.
	The negative part has convergence ray $\infty$.
\end{remark}

\begin{proof} 
	We look for a formal Laurent series solution to the $q$-difference equation (\ref{stokes:eqn_qde_laurent_series}).
	Let the Laurent series in the input $w$
	\[
		h_q(w) = \sum_{d \in \mathbb{Z}} h_d(q) w^d
	\]
	Let's assume the Laurent series $h_q$ is a solution of the $q$-difference equation (\ref{stokes:eqn_qde_laurent_series}). Then,
	\[
		\sum_{d \in \mathbb{Z}} \left( \alpha^{-1} P \right)^{N+1} q^{d(N+1)} h_d w^d
		=
		\sum_{d \in \mathbb{Z}} (-1)^{N+1} q^{N+1} 
		\left[
			\prod_{i=0}^N \Lambda_i \left(1- \alpha^{-1} P \Lambda_i^{-1} q^d \right)
		\right]
		h_d w^{d+1}
	\]
	Identifying the coefficients in front of $w^{d+1}$, we get the following recursion relation satisfied by the family of coefficients $\left(h_d\right)_{d \in \mathbb{Z}}$.
	\[
		\left( \alpha^{-1} P \right)^{N+1} q^{(d+1)(N+1)} h_{d+1}
		=
		(-1)^{N+1} q^{N+1}
		\left[
			\prod_{i=0}^N \Lambda_i \left(1- \alpha^{-1} P \Lambda_i^{-1} q^d \right)
		\right]
		h_d
	\]
	Recall that the $q$-Pochhammer symbol $(a;q)_d$ is a solution of the recursion equation $(a;q)_{d+1} = (1-aq^d) (a;q)_d$.
	Therefore, we get a solution $h_d$ of the previous recursion equation given by
	\begin{align*}
		h_{d+1}
		=&
		(-1)^{(N+1)(d+1)}
		\left(
			q^{N+1}
		\right)^{-\frac{d(d+1)}{2}}
		\left(
			\Lambda_0^{-1} \alpha^{-1} P,
			\dots,
			\Lambda_N^{-1} \alpha^{-1} P
			;q
		\right)_{d+1} \times
		\\
		&\times \left(
			\left( \alpha P^{-1} \right)^{N+1} \Lambda_0 \cdots \Lambda_N
		\right)^{d+1}
		h_0,
	\end{align*}
	where $h_0 \in \mathbb{C}$. Setting $h_0=1$ produces a solution which is also the bilateral $q$-hypergeometric series given by Equation (\ref{stokes:eqn_laurent_series_solution})
\end{proof}

We can now give the formula for our fundamental solution in the proposition below.

\begin{prop}\label{stokes:prop_formal_fund_sol}
	Consider the $q$-difference equation (\ref{stokes:eqn_qde_infinity}), given by
	\[
		\left[
		(-1)^{N+1} q^{N+1} \Lambda_0 \cdots \Lambda_N w
		\left( 1 - \Lambda_0^{-1} \qdeop{w} \right) \cdots \left( 1 - \Lambda_N^{-1} \qdeop{w} \right)
		- \left( \qdeop{w} \right)^{N+1}
		\right]
		g_q(w) = 0
	\]
	Assume that $\alpha \in \mathbb{C}^*-q^\mathbb{Z}$ and that for any $i \neq j \in \{0,\dots,N\}$, $\Lambda_i \Lambda_j^{-1} \notin q^\mathbb{Z}$.
	Denote by $h_q(w)$ the formal Laurent series (\ref{stokes:eqn_laurent_series_solution}).
	Then, the $q$-difference equation (\ref{stokes:eqn_qde_infinity}) admits a basis of formal solutions given by, for $i \in \{0, \dots, N\}$,
	\begin{align*}
		g_i(w)
		&=
		\left(
			e_{q,\alpha^{-1}P}(w) h_q(w)
		\right)_{|P=\Lambda_i}
		\\
		&=
		e_{q,\alpha^{-1} \Lambda_i}
		\tensor[_{N+1}]{\psi}{_0} \left(
            \begin{gathered}
                \Lambda_0^{-1} \alpha^{-1} \Lambda_i \quad \cdots \quad \Lambda_N^{-1} \alpha^{-1} \Lambda_i \\
                -
            \end{gathered}
            \quad \middle| \quad
            q, \left( \alpha \Lambda_i^{-1} \right)^{N+1} \Lambda_0 \cdots \Lambda_N w
        \right).
	\end{align*}
\end{prop}

\begin{remark}
	Before giving a proof of this statement, we point out that if $\alpha \in q^\mathbb{Z}$, then there exists a $d \in \mathbb{Z}$ such that $(\alpha^{-1};q)_d=0$, and therefore the expression $g_i$ is either undefined or is not a solution of the $q$-difference equation. 
\end{remark}

\begin{proof}
	Let $i \in \{0, \dots, N\}$. Let us show that the function $g_i$ is a formal solution of the $q$-difference equation (\ref{stokes:eqn_qde_infinity}).
	By setting $P=\Lambda_i$ in the statement of Lemma \ref{stokes:lemma_formal_fund_sol_ansatz_qde}, we can construct one solution by solving the $q$-difference equation (\ref{stokes:eqn_qde_laurent_series}), having replaced $P$ by $\Lambda_i$.
	A formal solution of this new $q$-difference equation can be found in Lemma \ref{stokes:lemma_ansatz_qde_solution} after setting $P=\Lambda_i$, which is precisely the function $g_i$.

	Assuming the condition that for any $i \neq j \in \{0,\dots,N\}$, $\Lambda_i \Lambda_j^{-1} \notin q^\mathbb{Z}$, we obtain that the functions $\left(g_i\right)_{i \in \{0, \dots, N\}}$ are independent over the field of $q$-constants $\mathcal{M}\left( \mathbb{E}_q \right)$.
\end{proof}

\subsubsection{Analytic solution}

\begin{defin}
	Let $f(w) = \sum_{d \geq 0} f_d w^d \in \mathbb{C}[ \! [ w ] \! ]$ be a formal power series.
	The $q$-Borel transform of the formal power series $f$ is given by the expression
	\[
		\mathcal{B}_q f(\xi)
		:=
		\sum_{d \geq 0}
		f_d
		q^{\frac{d(d-1)}{2}}
		\xi^d
		\in
		\mathbb{C}[ \! [ \xi ] \! ].
	\]
\end{defin}

\begin{defin}[\cite{Dreyfus_Eloy:qBorelLaplace}, Definition 1.2]
	Let $[ \lambda; q ] \in \mathbb{C}^*/q^\mathbb{Z}$ be a discrete $q$-spiral and $f \in \mathcal{M}\left( \mathbb{C}^*, 0 \right)$ be a germ of a meromorphic function with essential singularity at 0.
	We say the function $g$ admits a $q$-Laplace transform along the $q$-spiral $[ \lambda; q ]$ if there exists a constant $\varepsilon > 0$ and an domain $\Omega \subset \mathbb{C}$ such that
	\begin{enumerate}[label=(\roman*)]
		\item
		The domain $\Omega$ contains the domain
		\[
			\bigcup_{m \in \mathbb{Z}}
			\left\{
				\xi \in \mathbb{C}^*
				\, , \,
				|\xi - \lambda q^m|
				<
				\varepsilon | \lambda q^m|
			\right\}
			\subset
			\Omega.
		\]

		\item
		The function $g$ admits an analytic continuation $\bar{g}$ on the domain $\Omega$. Furthermore, we ask that there exists constants $C_1, C_2 > 0$ such that $\bar{f}$ satisfies the bound
		\[
			\left|\bar{f}(\xi)\right|
			<
			C_1
			\theta_{|q|}
			\left(
				C_2 |\xi|
			\right).
		\]
		A function satisfying such a bound will be said to have $q$-exponential growth at $\infty$.
	\end{enumerate}
	We will denote by $\mathcal{H}_q^{[ \lambda; q ]}$ the space of functions satisfying the conditions (i) and (ii).
\end{defin}

\begin{remark}
	Notice that the definitions we gave so far are concerned with power series, while the formal fundamental solution built in Proposition \ref{stokes:prop_formal_fund_sol} is a Laurent series.
	We formally extend the definition of the $q$-Laplace transform to Laurent series by setting
	\[
		\mathcal{B}_q \left(
			\sum_{d \in \mathbb{Z}} f_d w^d
		\right)
		(\xi)
		=
		\sum_{d \in \mathbb{Z}}
		f_d
		q^{\frac{d(d-1)}{2}}
		\xi^d
		\in
		\mathbb{C}[ \! [ \xi^{\pm 1} ] \! ].
	\]
	By doing so, there is a chance that the negative powers part of the Laurent series is no longer convergent, but it still is in the case of our fundamental solution.
\end{remark}

\begin{defin}
	Let $g \in \mathcal{H}_q^{[ \lambda; q ]}$ be a function admitting a $q$-Laplace transform along the $q$-spiral $[ \lambda; q ]$.
	We defined the $q$-Laplace transform of the function $f$ by the expression
	\[
		\mathcal{L}_q^{[ \lambda; q ]}g(w)
		:=
		\sum_{m \in \mathbb{Z}}
		\frac{
			g(\lambda q^m)
		}{
			\theta_q
			\left(
				\frac{\lambda q^m}{w}
			\right)
		}
		\in
		\mathcal{M} \left( \mathbb{C}^*, 0 \right).
	\]
\end{defin}

We will now define a $q$-Borel--Laplace sum. 

\begin{prop}[\cite{Dreyfus_Eloy:qBorelLaplace}, Lemma 1.5 and 1.7]
	Consider a convergent power series $f \in \mathbb{C}\{w\}$ and a $q$-spiral $[ \lambda; q ] \in \mathbb{C}/q^\mathbb{Z}$, then
	\[
		\mathcal{L}_q^{[ \lambda; q ]} \mathcal{B}_q f(w) = f(w).
	\]
\end{prop}

Note that this formula extends to formal Laurent series, as the recursion strategy used in the proof of Lemma 1.7 in \cite{Dreyfus_Eloy:qBorelLaplace} can be used to prove that for a fixed $l \in \mathbb{Z}$ and every $a \in \mathbb{C}$, $\mathcal{L}_q^{[ \lambda; q ]} \mathcal{B}_q aw^{l+1} = a w^{l+1}$, then $\mathcal{L}_q^{[ \lambda; q ]} \mathcal{B}_q aw^{l} = a w^{l}$ also holds.
Indeed, a computation gives the formulas
\begin{align*}
	&\mathcal{L}_q^{[ \lambda; q ]} \left( \xi  \mathcal{B}_q(f)(\xi) \right)
	=
	w q^{-w \partial w} \mathcal{L}_q^{[ \lambda; q ]} \mathcal{B}_q f(w)
	=
	\mathcal{L}_q^{[ \lambda; q ]} \mathcal{B}_q \left(	wq^{-w \partial_w} f(w) \right),
	\\
	&\mathcal{L}_q^{[ \lambda; q ]} \left( \xi^{-1}  \mathcal{B}_q(f)(\xi) \right)
	=
	q^{w \partial w} \left( w^{-1} \mathcal{L}_q^{[ \lambda; q ]} \mathcal{B}_q f(w) \right)
	=
	\mathcal{L}_q^{[ \lambda; q ]} \mathcal{B}_q \left(	q^{w \partial_w} \left(w^{-1} f(w)\right) \right).
\end{align*}

\begin{defin}\label{stokes:def_qborel_laplace}
	Let $f \in \mathbb{C}[\![w^{\pm 1}]\!]$ be a formal series.
	We say that the function $f$ is $q$-Borel--Laplace summable along the $q$-spiral $[ \lambda; q ]$ if it satisfies the condition
	\[
		\mathcal{B}_q f \in \mathcal{H}_q^{[ \lambda; q ]}
	\]
	For such a function, we define its $q$-Borel--Laplace resummation to be the function defined by
	\[
		\mathcal{S}_q^{[ \lambda; q ]}f(w)
		:=
		\left(
			\mathcal{L}_q^{[ \lambda; q ]} \mathcal{B}_q f
		\right)
		(w)
	\]
\end{defin}

\begin{prop}\label{stokes:prop_analytic_fund_sol}
	Let $\left(g_i \right)$ be the basis of formal fundamental solution of the $q$-difference equation (\ref{stokes:eqn_qde_infinity}) constructed in Proposition \ref{stokes:prop_formal_fund_sol}, and let $\mathcal{S}_q^{[\lambda;q]}$ denote the $q$-Borel--Laplace transform defined in Definition \ref{stokes:def_qborel_laplace}.
	Then, the family $\left( \mathcal{S}_{q^{N+1}}^{\left[\lambda;q^{N+1}\right]} g_i \right)_i$ is a fundamental solution of the $q$-difference (\ref{stokes:eqn_qde_infinity}). 
\end{prop}

	To prove such a statement, one has to check $q^{N+1}$-resummability of the bilateral $q$-series (\ref{stokes:eqn_qde_laurent_series}).
	This relies on an analytical continuation of the $q^{N+1}$-Borel transform, which is given in the coming Corollary \ref{stokes:coro_borel_an_continuation}.
	The proof that this analytical continuation has $q^{N+1}$-exponential growth along a domain $\Omega$ is exactly the same as in the case of unilateral $q$-hypergeometric series, which is given by \cite{Adachi_borel_sum}, Theorem 3.1.
	Indeed, one can check the analytical continuation in both cases can be written under the form, with $C_j, A_j \in \mathbb{C}$
	\[
		\bar{g}(\xi)
		= \sum_j C_j \frac{\theta_q\left(a_j \xi\right)}{\theta_q(\xi)}
		\tensor[_{s}]{\varphi}{_r} \left(
            \begin{gathered}
                c_1 \quad \cdots \quad c_s \\
                d_1 \quad \cdots \quad d_r
            \end{gathered}
            \middle| q, A_j \frac{1}{w}
        \right);
	\]
	for which Adachi's arguments contained in Section 5 of \cite{Adachi_borel_sum} apply identically.

\subsection{Connection numbers for quantum $K$-theory of projective spaces}

We will now compute a base change formula between the $J$-function and the fundamental solution at $Q=\infty$ built in Proposition \ref{stokes:prop_analytic_fund_sol}.

\begin{notation}
	Let $\underline{a} := (a_1, \dots, a_r) \in \mathbb{C}^r$ be a multi-index. For $d \in \mathbb{Z} \cup \{\infty\}$, $j \in \{1, \dots, r\}$ and $\gamma \in \mathbb{C}$ we will use the following notations:
	\begin{align*}
		(\underline{a};q)_d &:= (a_1, \dots, a_r;q)_d, 	\\
		(\gamma \underline{a};q)_d &:= (\gamma a_1, \dots, \gamma a_r;q)_d, 	\\
		\underline{a^{-1}} &:= (a_1^{-1}, \dots, a_r^{-1}), 	\\
		\widehat{\underline{a}}^{(j)} &:= (a_1, \dots, a_{j-1}, a_{j+1}, \dots, a_r) \in \mathbb{C}^{r-1}, 	\\
		\pi(\underline{a}) &:= a_1 \cdots a_r.
	\end{align*}
\end{notation}

Our goal is to prove the following computation.

%
%

\begin{thm}\label{stokes:thm_connection_numbers_for_qk}
	Write $\underline{\Lambda}:=\left( \Lambda_0, \dots, \Lambda_N \right)$.
	Let $\alpha \in \mathbb{C}^*-q^\mathbb{Z}$ and let $\left[ \lambda; q^{N+1} \right]$ be a $q^{N+1}$-spiral.
	Denote by $\left(g_k^{\left[ \lambda; q^{N+1} \right]}\right)_{k \in \{0,\dots,N\}}$ the fundamental solution of the $q$-difference equation for quantum $K$-theory at $\infty$ given by Proposition \ref{stokes:prop_analytic_fund_sol}:
	\[
		g_k(w)
		=
		\left\{
			e_{q,\alpha^{-1}P}(w)
		\left[
			\mathcal{S}^{\left[ \lambda; q^{N+1} \right]}_{q^{N+1}}
			\tensor[_{N+1}]{\psi}{_0} \left(
	            \begin{gathered}
	                \alpha^{-1} P \underline{\Lambda^{-1}}
	                \\
	                -
	            \end{gathered}
	            \, \middle| \,
	            q, \left( \alpha P^{-1} \right)^{N+1} \pi(\underline{\Lambda}) w
	        \right)
	    \right](w)
		\right\}_{|P=\Lambda_k}.
	\]
	Then, this fundamental solution at $\infty$ can by expressed with the fundamental solution at $0$ given by the small $J$-function as in the following identity.
	\[
		g_k^{\left[ \lambda; q^{N+1} \right]}(w)
		=
		\sum_{j=0
		}^N
		R_{k,j}^{\left[ \lambda; q^{N+1} \right]}(q,w)
		J^{K\textnormal{th, eq}}_{|P = \Lambda_j}\left(q,\frac{1}{w}\right),
	\]
	where $R_{k,j}^{\left[ \lambda; q^{N+1} \right]}$ is the $q^{N+1}$-constant function given by
	\begin{align*}
		R_{k,j}^{\left[ \lambda; q^{N+1} \right]}(q,w)
		=
        &\frac{
        	\left(
        		q, \frac{
        			\alpha^{-1} \Lambda_k
        		}
        		{
        			\widehat{\underline{\Lambda}}^j
        		} ; q
        	\right)_\infty
        }
        {
        	\left(
        		q \alpha \Lambda_k^{-1} \Lambda_j, \frac{\Lambda_j}{\widehat{\underline{\Lambda}}^j} ; q
        	\right)_\infty
        }
        \frac
       	{
	       	\theta_q \left
	       		((-1)^N \frac{\lambda \alpha^{-1} \Lambda_k}{\Lambda_j}
	       	\right)
       	}
       	{
	       	\theta_q \left(
	       		(-1)^N \lambda
	       	\right)
        }
        \frac
       	{
	       	\theta_{q^{N+1}} \left(
	       		\frac{
	       			\lambda \Lambda_j^{N+1}
	       		}
	       		{
	       			\pi(\underline{\Lambda}) w
	       		}
	       	\right)
       	}
       	{
	       	\theta_{q^{N+1}} \left(
	       		\frac{
	       			\lambda
	       		}
	       		{
	       			\left( \alpha \Lambda_k^{-1} \right)^{N+1} \pi(\underline{\Lambda}) w
	       		}
	       	\right)
        }
        \times
        \\
        &\times
        e_{q,\alpha^{-1}\Lambda_k}(w)
        \Lambda_j^{\ell_q\left( \frac{1}{w} \right)}
	\end{align*}
\end{thm}

Our strategy to prove this theorem will be the same as the one found in \cite{Adachi_borel_sum}: we start from a connection number for a regular singular bilateral $q$-hypergeometric series, identify some limit of these connection numbers as an identity between $q$-Borel transforms, then apply a $q$-Laplace transform to the identity.

%
%

\begin{prop}[\cite{Slater_hypergeometric}, Equation 5.2.4 p.165; see also \cite{Chan_Sears_stuff}, Theorem 2.1]\label{stokes:lemma_sears_connection_formula}
	Let $\underline{a}, \underline{b} \in \mathbb{C}^r$.
	Assuming the following series are finite sums, or assuming $\left| \frac{\pi(\underline{b})}{\pi(\underline{a})} \right| < |z| < 1$,
	\begin{align*}
		\tensor[_{r}]{\psi}{_r} \left(
            \begin{gathered}
                \underline{a} \\
                \underline{b}
            \end{gathered}
            \,\middle|\, q, w
        \right)
        =
        \sum_{j=1}^r
        C_j(q)
        \frac
        	{
	        	\left(
	        		a_j w,\frac{q}{a_j w} ; q
	        	\right)_\infty
        	}
        	{
	        	\left(
	        		w,\frac{q}{w} ; q
	        	\right)_\infty
        	}
        \tensor[_{r}]{\varphi}{_{r-1}} \left(
            \begin{gathered}
                \frac{a_j q}{\underline{b}} \\
                \frac{a_j q}{\widehat{\underline{a}}^j}
            \end{gathered}
            \,\middle|\, q, \frac{\pi(\underline{b})}{\pi(\underline{a})w}
        \right),
	\end{align*}
	where
	\begin{align*}
		C_j(q)
		&:=
		\frac
		{
			\left(
				q, \widehat{\underline{a}}^j, \frac{\underline{b}}{a_j} ; q
			\right)_\infty
		}
		{
			\left(
				\frac{q}{a_j}, \frac{\widehat{\underline{a}}^j}{a_j},\underline{b} ; q
			\right)_\infty
		}
		\in \mathbb{C},
		\\
		\frac{a_j q}{\underline{b}}
		&:=
		\left(
			\frac{a_j q}{b_1}, \dots, \frac{a_j q}{b_r}
		\right)
		\in \mathbb{C}^r.
	\end{align*}
\end{prop}

\begin{remark}
	We have
	\[
		\frac
        	{
	        	\left(
	        		a_jz,\frac{q}{a_j z} ; q
	        	\right)_\infty
        	}
        	{
	        	\left(
	        		z,\frac{q}{z} ; q
	        	\right)_\infty
        }
        =
        \frac{\theta_q(-a_j z)}{\theta_q(-z)}
	\]
	Over the field of $q$-constants $\mathcal{M}\left( \mathbb{E}_q \right)$, the corresponding function is linearly equivalent to the function given by $a_j^{-\ell_q(z)}$.
\end{remark}

%
%

By taking the limit $\underline{b} \to \underline{0}$ in the identity of Proposition \ref{stokes:lemma_sears_connection_formula}, we obtain the following corollary.

\begin{coro}\label{stokes:coro_borel_an_continuation}
	We have the following identity of analytic functions
	\[
		\tensor[_{r}]{\psi}{_r} \left(
            \begin{gathered}
                \underline{a} \\
                \underline{0}
            \end{gathered}
            \,\middle|\, q, w
        \right)
        =
        \sum_{j=1}^r
        C_j'(q)
        \frac
        	{
	        	\theta_q(-a_j w)
        	}
        	{
	        	\theta_q(-w)
        	}
        \tensor[_{0}]{\varphi}{_{r-1}} \left(
            \begin{gathered}
                - \\
                \frac{a_j q}{\widehat{\underline{a}}^j}
            \end{gathered}
            \,\middle|\, q, \frac{q^r a_j^{r-1}}{\pi(\widehat{\underline{a}}^j)w}
        \right),
	\]
	where
	\[
		C'_j(q)
		:=
		\left(
			q, \widehat{\underline{a}}^j;q
		\right)_\infty
		\left(
			\frac{q}{a_j}, \frac{\widehat{\underline{a}}^j}{a_j};q
		\right)_\infty^{-1}
	\]
\end{coro}

\begin{remark}
	The motivation for this using this corollary is the observation that, denoting $\mathcal{B}_q$ the $q$-Borel transform,
	\[
		\mathcal{B}_{q^r}
		\tensor[_{r}]{\psi}{_0} \left(
            \begin{gathered}
                \underline{a} \\
                -
            \end{gathered}
            \,\middle|\, q, w
        \right)
        =
        \tensor[_{r}]{\psi}{_r} \left(
            \begin{gathered}
                \underline{a} \\
                \underline{0}
            \end{gathered}
            \,\middle|\, q, (-1)^rw
        \right)
	\]
	We also notice that the statement of this corollary does not make sense if we were to set $\underline{a} = \underline{1}$, e.g. if we were doing equivariant limit in equivariant quantum $K$-theory.
\end{remark}

\begin{proof}
	We have $\lim_{b_i \to 0} C_j(q) = C_j'(q)$ by the convention $(0;q)_d = 1$.
	The remaining computation relies on the observation that
	\[
		\lim_{b_i \to 0}
		\left(
			\frac{a_j q}{b_i} ; q
		\right)_d
		b_i^d
		=
		\lim_{b_i \to 0}
		\prod_{l=1}^d \left( b_i - a_j q^l \right)
		=
		(-1)^d a_j^d q^{\frac{d(d+1)}{2}}
	\]
	Therefore, we have
	\begin{align*}
		\tensor[_{r}]{\varphi}{_{r-1}} \left(
            \begin{gathered}
                \frac{a_j q}{\underline{b}} \\
                \frac{a_j q}{\widehat{\underline{a}}^j}
            \end{gathered}
            \,\middle|\, q, \frac{\pi(\underline{b})}{\pi(\underline{a})w}
        \right)
        &=
        \sum_{d \geq 0}
        \frac{
        	\left(
        		\frac{a_j q}{\underline{b}};q
        	\right)_d
        	\pi(\underline{b})^d
        }
        {
        	\left(
        		\frac{a_j q}{\widehat{\underline{a}}^j};q
        	\right)_d
        }
        \left(
        	\frac{1}{\pi(\underline{a})^dz^d}
        \right)
        \\
        &\to
        \sum_{d \geq 0}
        \frac{
        	(-1)^{dr} a_j^{dr} q^{r \frac{d(d-1)}{2}} q^{dr}
        }
        {
        	\left(
        		\frac{a_j q}{\widehat{\underline{a}}^j};q
        	\right)_d
        }
        \left(
        	\frac{1}{\pi(\underline{a})^dz^d}
        \right)
        \\
        &=
        \sum_{d \geq 0}
        \frac{
        	1
        }
        {
        	\left(
        		\frac{a_j q}{\widehat{\underline{a}}^j};q
        	\right)_d
        }
        \left(
        	(-1)^d q^{\frac{d(d-1)}{2}}
        \right)^r
        \left(
        	\frac{q^r a_j^{r-1}}{\pi(\widehat{\underline{a}}^j)z}
        \right)^d
        \\
        &=
        \tensor[_{0}]{\varphi}{_{r-1}} \left(
            \begin{gathered}
                - \\
                \frac{a_j q}{\widehat{\underline{a}}^j}
            \end{gathered}
            \,\middle|\, q, \frac{q^r a_j^{r-1}}{\pi(\widehat{\underline{a}}^j)w}
        \right)
	\end{align*}
\end{proof}

%
%

\begin{coro}\label{stokes:coro_transformation_formula}
	Let $\mathcal{L}_{q^r}^{[\lambda;q^r]}$ be the $q^r$-Laplace transform along the $q$-spiral $[\lambda;q^r]$.
	We have the following identity of analytic functions
	\begin{align*}
		&\left[
			\mathcal{L}_{q^r}^{[\lambda;q^r]} \tensor[_{r}]{\psi}{_r} \left(
	            \begin{gathered}
	                \underline{a} \\
	                \underline{0}
	            \end{gathered}
	            \,\middle|\, q, (-1)^r w
	        \right)
	    \right]
	    (x)
	    \\
        &=
        \sum_{j=1}^r
        C_j'(q)
        \frac
        	{
	        	\theta_q \left(
	        		(-1)^{r+1} a_j \lambda
	        	\right)
        	}
        	{
	        	\theta_q \left(
	        		(-1)^{r+1} \lambda
	        	\right)
        	}
        	\frac
        	{
	        	\theta_{q^r} \left(
	        		\frac{\lambda}{a_j^r x}
	        	\right)
        	}
        	{
	        	\theta_{q^r} \left(
	        		\frac{\lambda}{x}
	        	\right)
        	}
        \tensor[_{r}]{\varphi}{_{r-1}} \left(
            \begin{gathered}
                \underline{0} \\
                \frac{a_j q}{\widehat{\underline{a}}^j}
            \end{gathered}
            \,\middle|\, q, \frac{1}{\pi(\underline{a})x}
        \right)
	\end{align*}
\end{coro}

We recall that the $q^r$-Laplace transform along the $q$-spiral $[\lambda;q^r]$ of a function $g$ is given by
\[
	\left[
		\mathcal{L}_{q^r}^{ [ \lambda ; q^r ] } g
	\right]
	(x)
	:=
	\sum_{n \in \mathbb{Z}}
	\frac{
		g(\lambda q^{rn})
	}
	{
		\theta_{q^r}\left({\frac{\lambda q^{rn}}{x}}\right)
	}
\]
The main idea of the proof of the corollary is to make a change of variable for the summation on the index $n \in \mathbb{Z}$ to make the Laurent series of the function $\theta_{q^r}$ appear.

\begin{proof}
	In the expression for the $q^r$-Laplace transform of the right hand side of Corollary \ref{stokes:coro_transformation_formula},
	we use the identity (deduced from the $q$-difference equation satisfied by the theta function)
	\[
		\theta_q \left(q^r x \right)
		=
		\frac{1}
		{
			q^{\frac{r(r-1)}{2}} x^r
		}
		\theta_q(x)
	\]
	We therefore obtain
	\begin{align*}
		&\left[
			\mathcal{L}_{q^r}^{[\lambda;q^r]} \tensor[_{r}]{\psi}{_r} \left(
	            \begin{gathered}
	                \underline{a} \\
	                \underline{0}
	            \end{gathered}
	            \,\middle|\, q, (-1)^{r+1} w
	        \right)
	    \right]
	    (x)
        \\
        &=
        \sum_{j=1}^r \sum_{n \in \mathbb{Z}}
        C'_j(q)
        \frac{1}{a_j^{rn}}
        \frac{
        	\theta_q \left(
        		(-1)^{r+1} a_j \lambda
        	\right)
        }
        {
        	\theta_q \left(
        		(-1)^{r+1} \lambda
        	\right)
        }
        \left(
        	\frac{\lambda}{x}
        \right)^{n}
        q^{r \frac{n(n-1)}{2}}
        \frac{1}{
        	\theta_{q^r} \left(
        		\frac{\lambda}{x}
        	\right)
        }
        \times
        \\
        &\times
        \sum_{d \geq 0}
        \frac{1}{
        	\left(
        		\frac{a_j q}{\widehat{\underline{a}}^j} ; q
        	\right)_d
        }
        q^{r \frac{d(d-1)}{2}}
        \frac{q^{rd}}{q^{drn}}
        \frac{
        	a_j^{(r-1)d}
        }
        {
        	\pi \left(
        		\widehat{\underline{a}}^j
        	\right)^d
        	\lambda^d
        }
	\end{align*}
	Multiplying all the terms of the form $q^{( \textnormal{exponent})}$ together, we obtain
	\[
		q^{
		r \frac{n(n-1)}{2}
		+
		r \frac{r(r-1)}{2}
		+
		dr
		-drn
		}
		=
		q^{r \frac{(n-d)(n-d-1)}{2}}
	\]
	Setting $n'=n-d$, we have
	\begin{align*}
		&\sum_{n \in \mathbb{Z}}
		\frac{1}{a_j^{rn}}
		\left(
        	\frac{\lambda}{x}
        \right)^{n}
		q^{r \frac{(n-d)(n-d-1)}{2}}
		\\
		&=
		\left(
			\sum_{n' \in \mathbb{Z}}
			\left(
				\frac{\lambda}{a_j^r x}
			\right)^{n'}
			q^{r \frac{(n')(n'-1)}{2}}
		\right)
		\left(
			\frac{\lambda}{a_j^r x}
		\right)^d
		\\
		&=
		\theta_{q^r}
		\left(
			\frac{\lambda}{a_j^r x}
		\right)
		\left(
			\frac{\lambda}{a_j^r x}
		\right)^d
	\end{align*}
	Therefore,
	\begin{align*}
		&\left[
			\mathcal{L}_{q^r}^{[\lambda;q^r]} \tensor[_{r}]{\psi}{_r} \left(
	            \begin{gathered}
	                \underline{a} \\
	                \underline{0}
	            \end{gathered}
	            \,\middle|\, q, (-1)^{r+1} w
	        \right)
	    \right]
	    (x)
        \\
        &=
        \sum_{j=1}^r
        C'_j(q)
        \frac{
        	\theta_q \left(
        		(-1)^{r+1} a_j \lambda
        	\right)
        }
        {
        	\theta_q \left(
        		(-1)^{r+1} \lambda
        	\right)
        }
        \frac{
        	\theta_{q^r}
			\left(
				\frac{\lambda}{a_j^r x}
			\right)
        }
        {
        	\theta_{q^r}
			\left(
				\frac{\lambda}{x}
			\right)
	    }
	    \sum_{d \geq 0}
	    \frac{1}{
        	\left(
        		\frac{a_j q}{\widehat{\underline{a}}^j} ; q
        	\right)_d
        }
        \left(
			\frac{\lambda}{a_j^r x}
		\right)^d
		\frac{
        	a_j^{(r-1)d}
        }
        {
        	\pi \left(
        		\widehat{\underline{a}}^j
        	\right)^d
        	\lambda^d
        }
	\end{align*}
	We conclude recognizing the $q$-hypergeometric series
	\[
		\sum_{d \geq 0}
	    \frac{1}{
        	\left(
        		\frac{a_j q}{\widehat{\underline{a}}^j} ; q
        	\right)_d
        }
        \left(
			\frac{\lambda}{a_j^r x}
		\right)^d
		\frac{
        	a_j^{(r-1)d}
        }
        {
        	\pi \left(
        		\widehat{\underline{a}}^j
        	\right)^d
        	\lambda^d
        }
        =
        \tensor[_{r}]{\varphi}{_{r-1}} \left(
            \begin{gathered}
                \underline{0} \\
                \frac{a_j q}{\widehat{\underline{a}}^j}
            \end{gathered}
            \,\middle|\, q, \frac{1}{\pi(\underline{a})x}
        \right)
	\]
\end{proof}

Applying this corollary to the case of quantum $K$-theory gives the following statement below.

\begin{coro}
	Write $\underline{\Lambda}:= (\Lambda_0, \dots, \Lambda_N) \in K_{T^{N+1}}\left( \mathbb{P}^N \right)^{N+1}$.
	Let $\alpha \in \mathbb{C}^*-q^\mathbb{Z}$ and let $\left[ \lambda; q^{N+1} \right]$ be a $q^{N+1}$-(discrete) spiral.
	We have the following identity of functions
	\begin{align*}
		&e_{q,\alpha^{-1}P}(w)
		\left[
			\mathcal{S}^{\left[ \lambda; q^{N+1} \right]}_{q^{N+1}}
			\tensor[_{N+1}]{\psi}{_0} \left(
	            \begin{gathered}
	                \underline{\Lambda}^{-1} \alpha^{-1} P
	                \\
	                -
	            \end{gathered}
	            \, \middle| \,
	            q, \left( \alpha P^{-1} \right)^{N+1} \pi(\underline{\Lambda}) w
	        \right)
	    \right](w)
        \\
        &=
        \sum_{j=0}^N
        \frac{
        	\left(
        		q, \frac{
        			\alpha^{-1}P
        		}
        		{
        			\widehat{\underline{\Lambda}}^j
        		} ; q
        	\right)_\infty
        }
        {
        	\left(
        		q \alpha P^{-1} \Lambda_j, \frac{\Lambda_j}{\widehat{\underline{\Lambda}}^j} ; q
        	\right)_\infty
        }
        \frac
       	{
	       	\theta_q \left
	       		((-1)^N \frac{\lambda \alpha^{-1} P}{\Lambda_j}
	       	\right)
       	}
       	{
	       	\theta_q \left(
	       		(-1)^N \lambda
	       	\right)
        }
        \frac
       	{
	       	\theta_{q^{N+1}} \left(
	       		\frac{
	       			\lambda \Lambda_j^{N+1}
	       		}
	       		{
	       			\pi(\underline{\Lambda}) w
	       		}
	       	\right)
       	}
       	{
	       	\theta_{q^{N+1}} \left(
	       		\frac{
	       			\lambda
	       		}
	       		{
	       			\left( \alpha P^{-1} \right)^{N+1} \pi(\underline{\Lambda}) w
	       		}
	       	\right)
        }
        \times
        \\
        &\times
        e_{q,\alpha^{-1}P}(w)
        \Lambda_j^{\ell_q\left( \frac{1}{w} \right)}
        \left\{
        	\Lambda_j^{-\ell_q\left( \frac{1}{w} \right)}
        	\tensor[_{N+1}]{\varphi}{_N} \left(
	            \begin{gathered}
	                \underline{0}
	                \\
	                q \Lambda_j^{-1} \widehat{\underline{\Lambda}}^j
	            \end{gathered}
	            \, \middle| \,
	            q, \frac{1}{w}
	        \right)
        \right\}
	\end{align*}
\end{coro}
Notice that in the right hand side of the above identity, the function between the curly brackets is the small $J$-function:
\[
	\Lambda_j^{-\ell_q\left( \frac{1}{w} \right)}
   	\tensor[_{N+1}]{\varphi}{_N} \left(
        \begin{gathered}
            \underline{0}
            \\
            q \Lambda_j^{-1} \widehat{\underline{\Lambda}}^j
        \end{gathered}
        \, \middle| \,
        q, \frac{1}{w}
    \right)
    =
    J^{K\textnormal{th, eq}}_{|P = \Lambda_j}\left(q,\frac{1}{w}\right).
\]
From this observation, we obtain the identity announced in Theorem \ref{stokes:thm_connection_numbers_for_qk}.

\begin{remark}
	If we try to obtain a non equivariant version of the formula in Theorem \ref{stokes:thm_connection_numbers_for_qk}, the formula does not make sense as we no longer have bases of solutions on left and right hand sides.
	Nonetheless, let us consider the ring $K_{T^{N+1}}\left( \mathbb{P}^N \right) \otimes K_{T^{N+1}}\left( \mathbb{P}^N \right)$, denoting by $P_{(0)}$ (resp. $P_{(\infty)}$) the generator on the left (resp. right) factor.
	We introduce the equivariant $K$-theoretic number
	\begin{equation*}
	\begin{aligned}
		\textbf{R}^{\textnormal{eq}}(q,w)
		&:= \frac{
				\left(\alpha^{-1}P_{(\infty)}\underline{\Lambda^{-1}} ;q\right) 
			}{
				\left(\alpha^{-1}P_{(0)}\underline{\Lambda^{-1}} ;q\right)
			}
			\frac{
				(q;q)_\infty^2
			}{
			\left(q \alpha {P_{(\infty)}}^{-1}P_{(0)} ,\alpha^{-1}P_{(\infty)}{P_{(0)}}^{-1} ;q\right)
			}
			\frac
       	{
	       	\theta_q \left
	       		((-1)^N \frac{\lambda \alpha^{-1} {P_{(\infty)}} }{{P_{(0)}}}
	       	\right)
       	}
       	{
	       	\theta_q \left(
	       		(-1)^N \lambda
	       	\right)
        }
        \times
        \\
        &\times
        \frac
       	{
	       	\theta_{q^{N+1}} \left(
	       		\frac{
	       			\lambda {P_{(0)}}^{N+1}
	       		}
	       		{
	       			\pi(\underline{\Lambda}) w
	       		}
	       	\right)
       	}
       	{
	       	\theta_{q^{N+1}} \left(
	       		\frac{
	       			\lambda
	       		}
	       		{
	       			\left( \alpha {P_{(\infty)}}^{-1} \right)^{N+1} \pi(\underline{\Lambda}) w
	       		}
	       	\right)
        }
        e_{q,\alpha^{-1}{P_{(\infty)}}}(w)
        {P_{(0)}}^{\ell_q\left( \frac{1}{w} \right)}
		\\
		& \in K_{T^{N+1}}\left( \mathbb{P}^N \right) \otimes K_{T^{N+1}}\left( \mathbb{P}^N \right).
	\end{aligned}
\end{equation*}
Then, one can notice that $\textbf{R}(q,w) _ {\left| {P_{(\infty)}} = \Lambda_k, {P_{(0)}} = \Lambda_j  \right.} = R_{k,j}$, where $R_{k,j}$ is the equivariant connection number of Theorem \ref{stokes:thm_connection_numbers_for_qk}.
The non equivariant limit of the number $\textbf{R}(q,w)$ is well defined and given by
\begin{equation}\label{eqn:qmonodromy_tentative_non_equivariant_connection_numbers}
	\begin{aligned}
	\lim_{\underline{\Lambda} \to 1}
	\textbf{R}(q,w)
	&=
	\frac{
		\left(\alpha^{-1}P_{(\infty)} ;q \right)_\infty^{N+1}
	}{
		\left(\alpha^{-1}P_{(0)} ;q \right)_\infty^{N+1}
	}
	\frac{
		\left(q ;q \right)_\infty^{2}
	}{
		\left(q \alpha {P_{(\infty)}}^{-1}{P_{(0)}},\alpha^{-1} {P_{(\infty)}}{P_{(0)}}^{-1} ;q \right)_\infty
	}
	\frac{
	    \theta_q \left
	    	((-1)^N \frac{\lambda \alpha^{-1} {P_{(\infty)}} }{{P_{(0)}}}
	    \right)
       	}{
	       	\theta_q \left(
	       		(-1)^N \lambda
	       	\right)
        }
	\times
        \\
        &\times
        \frac
       	{
	       	\theta_{q^{N+1}} \left(
	       		\frac{
	       			\lambda {P_{(0)}}^{N+1}
	       		}
	       		{
	       			w
	       		}
	       	\right)
       	}
       	{
	       	\theta_{q^{N+1}} \left(
	       		\frac{
	       			\lambda
	       		}
	       		{
	       			\left( \alpha {P_{(\infty)}}^{-1} \right)^{N+1} w
	       		}
	       	\right)
        }
        e_{q,\alpha^{-1}{P_{(\infty)}}}(w)
        {P_{(0)}}^{\ell_q\left( \frac{1}{w} \right)}
	\end{aligned}
\end{equation}
We recall that a basis of solution in the non equivariant case is obtained by taking in the formula for the $J$-function the coefficient in front of $(1-P^{-1})^j$ for $j=0,\dots,n$.
Therefore, we may expect the connection numbers in the non equivariant case to be obtained by looking at the coefficients in front of $\left( 1 - {P_{(\infty)}} \right)^k \otimes \left( 1 - {P_{(0)}}\right)^j$ in the right hand side of Equation \eqref{eqn:qmonodromy_tentative_non_equivariant_connection_numbers}, once it is decomposed in this basis of $K\left( \mathbb{P}^N \right)^{\otimes 2}$.
Unfortunately, we are currently not able to write the identity Theorem \ref{stokes:thm_connection_numbers_for_qk} without any choice of basis in equivariant $K$-theory, thus we are not able to make such a non equivariant limit.
\end{remark}

%
%

\bibliographystyle{alpha}
\small\bibliography{Bibliography}

\end{document}